\documentstyle[10pt,amscd,xypic,amssymb]{amsart}

\xyoption{all}
\CompileMatrices

\newcommand{\nc}{\newcommand}

\makeatletter
\@addtoreset{equation}{section}
\makeatother

\newenvironment{proof}{{\noindent \textbf{Proof}\,\,}}{\hspace*{\fill}$\Box$\medskip}

\newtheorem{theorem}[subsection]{Theorem}

\newtheorem{proposition}[subsection]{Proposition}
\newtheorem{lemma}[subsection]{Lemma}

\newtheorem{conjecture}[subsection]{Conjecture}

\theoremstyle{definition}

\theoremstyle{remark}

\newtheorem{remark}[subsection]{Remark}

\nc{\fa}{{\mathfrak{a}}}
\nc{\fb}{{\mathfrak{b}}}
\nc{\fg}{{\mathfrak{g}}}
\nc{\fh}{{\mathfrak{h}}}
\nc{\fj}{{\mathfrak{j}}}
\nc{\fn}{{\mathfrak{n}}}
\nc{\fm}{{\mathfrak{m}}}
\nc{\fu}{{\mathfrak{u}}}
\nc{\fp}{{\mathfrak{p}}}
\nc{\fr}{{\mathfrak{r}}}
\nc{\ft}{{\mathfrak{t}}}
\nc{\fsl}{{\mathfrak{sl}}}
\nc{\fgl}{{\mathfrak{gl}}}
\nc{\hsl}{{\widehat{\mathfrak{sl}}}}
\nc{\hgl}{{\widehat{\mathfrak{gl}}}}
\nc{\hg}{{\widehat{\mathfrak{g}}}}
\nc{\chg}{{\widehat{\mathfrak{g}}}{}^\vee}
\nc{\hn}{{\widehat{\mathfrak{n}}}}
\nc{\chn}{{\widehat{\mathfrak{n}}}{}^\vee}

\nc{\Mod}{{\textrm{Mod}}}

\nc{\wGL}{{\widehat{GL}^+}}

\nc{\BA}{{\mathbb{A}}}
\nc{\BC}{{\mathbb{C}}}
\nc{\BM}{{\mathbb{M}}}
\nc{\BN}{{\mathbb{N}}}
\nc{\BF}{{\mathbb{F}}}
\nc{\BH}{{\mathbb{H}}}
\nc{\BP}{{\mathbb{P}}}
\nc{\BR}{{\mathbb{R}}}
\nc{\BZ}{{\mathbb{Z}}}

\nc{\kk}{{\mathbb{K}}}
\nc{\kko}{{\mathbb{K}}}

\nc{\coh}{{\text{Coh}}}

\nc{\CA}{{\mathcal{A}}}
\nc{\CC}{{\mathcal{C}}}
\nc{\CB}{{\mathcal{B}}}
\nc{\DD}{{\mathcal{D}}}
\nc{\CE}{{\mathcal{E}}}
\nc{\CF}{{\mathcal{F}}}
\nc{\tCF}{{\widetilde{\CF}}}
\nc{\tCT}{{\widetilde{\CT}}}
\nc{\oCF}{{\overline{\CF}}}
\nc{\CG}{{\mathcal{G}}}
\nc{\CL}{{\mathcal{L}}}
\nc{\CK}{{\mathcal{K}}}
\nc{\CI}{{\mathcal{I}}}
\nc{\CM}{{\mathcal{M}}}
\nc{\CH}{{\mathcal{H}}}
\nc{\CN}{{\mathcal{N}}}
\nc{\CO}{{\mathcal{O}}}
\nc{\CP}{{\mathcal{P}}}
\nc{\CR}{{\mathcal{R}}}
\nc{\CQ}{{\mathcal{Q}}}
\nc{\CS}{{\mathcal{S}}}
\nc{\CT}{{\mathcal{T}}}
\nc{\CU}{{\mathcal{U}}}
\nc{\CV}{{\mathcal{V}}}
\nc{\CW}{{\mathcal{W}}}

\nc{\tpsi}{{\widetilde{\Psi}}}
\nc{\Ker}{{\text{Ker }}}

\nc{\CX}{{\mathcal{X}}}
\nc{\tCX}{{\widetilde{\mathcal{X}}}}
\nc{\CY}{{\mathcal{Y}}}
\nc{\tCY}{{\widetilde{\mathcal{Y}}}}

\nc{\tN}{{\widetilde{\CN}}}
\nc{\pN}{{\BP\widetilde{\CN}}}

\nc{\tT}{{T}}

\nc{\fC}{{\mathfrak{C}}}
\nc{\fZ}{{\mathfrak{Z}}}
\nc{\fU}{{\mathfrak{U}}}
\nc{\fV}{{\mathfrak{V}}}
\nc{\fS}{{\mathfrak{S}}}

\nc{\od}{{\overline{d}}}
\nc{\rg}{{\textrm{R}\Gamma}}
\nc{\erg}{{\emph{R}\Gamma}}
\nc{\id}{{\textrm{Id}}}
\nc{\rhom}{{\textrm{RHom}}}

\def\ph{\varphi}

\def\Ext{\textrm{Ext}}
\def\Hom{\textrm{Hom}}

\def\e{\varepsilon}

\def\pol{\textrm{Poles}}

\def\mat{\textrm{Mat}}

\def\and{\textrm{ }\&\textrm{ }}

\def\sym{\textrm{Sym}}

\def\im{\textrm{Im }}

\def\tCF{\widetilde{\CF}}
\def\tCW{\widetilde{\CW}}
\def\tCW{\widetilde{\CW}}

\def\L{\Lambda}

\def\slz{\textrm{SL}_2(\BZ)}

\def\loccit{\emph{loc. cit. }}

\def\syt{\textrm{SYT}}
\def\sytx{\textrm{SYTx}}

\begin{document}

\title[Moduli of flags of sheaves and their $K-$theory]{\large{\textbf{Moduli of flags of sheaves and their $K-$theory}}}

\author[Andrei Negut]{Andrei Negut}
\address{Columbia University, Department of Mathematics, New York, NY, USA}
\address{Simion Stoilow Institute of Mathematics, Bucharest, Romania}
\email{andrei.negut@@gmail.com}

\maketitle

\begin{abstract}

We introduce moduli spaces of flags of sheaves on $\BP^2$, and use them to obtain functors between the derived categories of the usual moduli spaces of sheaves on $\BP^2$. These functors induce an action of the shuffle algebra on $K-$theory, which we reinterpret in terms of tautological classes. In particular, this action provides a $K-$theoretic version of Baranovsky's operators from \cite{Ba}.


\end{abstract}

\section{Introduction}

Through the work of Grojnowski and Nakajima (\cite{G} and \cite{Na}, for $r=1$) and Baranovsky (\cite{Ba}, for general $r$), there is an action of the Heisenberg Lie algebra $\hgl_1$ on the equivariant cohomology group $H$ of the moduli space $\CM$ of rank $r$ sheaves on a surface. This action is given by the correspondences:
\begin{equation}
\label{eqn:bgn}
\{(\CF,\CF') \text{ such that } \CF \supset \CF'\} \subset \CM \times \CM
\end{equation}
Later, Feigin-Tsymbaliuk (\cite{FT}) and Schiffmann-Vasserot (\cite{SV}) introduced an action of a certain larger algebra $\CA$ on the equivariant $K-$theory group of $\CM$. This algebra is known by many names: \\

\begin{itemize}

\item the double shuffle algebra, \\

\item the Hall algebra of an elliptic curve, \\

\item the doubly-deformed $W_{1+\infty}-$algebra, \\

\item the spherical part of Cherednik's DAHA, \\

\item $U_q(\widehat{\widehat{\mathfrak{gl}}}_1)$. 

\end{itemize}

\textbf{} \\
We will mostly work with the presentation of $\CA$ as a double shuffle algebra. As shown in \cite{Shuf}, the standard generators $u_{k,d}$ of the elliptic Hall algebra correspond to shuffle elements $P_{k,d} \in \CA$ which are given by explicit formulas. When $d=0$, these elements give rise to an action of the $q-$Heisenberg algebra $U_q(\hgl_1)$ on $K$ that deforms the construction of Baranovsky-Grojnowski-Nakajima. 

\textbf{} \\
In this paper, we will interpret these shuffle elements geometrically. The correspondences \eqref{eqn:bgn} will not be suitable for our purposes because they are too singular and contain too little geometric information about $K-$theory. Instead, we will work with the moduli spaces of flags of sheaves $\fZ^k$, whose points are chains of sheaves:
\begin{equation}
\label{eqn:mich}
(\CF_0 \supset ... \supset \CF_k)
\end{equation}
such that the successive quotients are all skyscraper sheaves supported at the same point. As we will see in Section \ref{sec:fine}, these varieties define functors between the derived categories of coherent sheaves of $\CM$. Let us emphasize that these functors are not simply compositions of $k$ usual Nakajima correspondences (see Remarks \ref{rem:duh0} and \ref{rem:duh}). We hope that these functors may be described by something like a categorical shuffle algebra, but we do not know how to define such a categorification. 

\textbf{} \\
The situation is much more clear at the level of $K-$theory, where the linear maps defined by the correspondences $\fZ^k$ are shown to be related to the shuffle algebra in Theorem \ref{thm:main}. In particular, the elements $P_{k,d} \in \CA$ act by some tautological bundles on the correspondence $\fZ^k$. Our formulas allow us to explicitly compute the matrix coefficients of $P_{k,d}$ in the basis of torus fixed points. Certain special cases of these coefficients are interpreted as refined knot invariants in \cite{GN}, where many combinatorial and representation theoretic consequences are explored.

\textbf{} \\
On a deeper level, the moduli spaces $\fZ^k$ of flags of sheaves \eqref{eqn:mich} appear in a conjecture of Bezrukavnikov and Okounkov concerning filtrations on the category of coherent sheaves on the Hilbert scheme. As a special case of this conjecture, they are expected to be related to certain modules of the rational Cherednik algebra (see \cite{GN}). We intend to develop the structure of $\fZ^k$ in more detail in subsequent papers. Let us say a few words about the structure of the present paper: \\

\begin{itemize}

\item In Section \ref{sec:geom} we present the well-known moduli space of sheaves on $\BP^2$ \\

\item In Section \ref{sec:fine} we introduce the moduli space of flags of sheaves on $\BP^2$, and use it to construct correspondences on the usual moduli of sheaves \\

\item In Section \ref{sec:shuf} we present the shuffle algebra $\CA$ and its action on $K$. We show that the generators $P_{k,d} \in \CA$ act via the above correspondences $\fZ^k$ \\

\item In Section \ref{sec:fixed} we compute the coefficients of the operators $P_{k,d}$ in the basis of torus fixed points \\

\item In Section \ref{sec:extra} we present the well-known $\Ext$ bundle $E$ and the Lagrangian correspondences $\fV^k$, and compute the linear maps they induce on $K$ 


\end{itemize}

\textbf{} \\
I would foremost like to thank my advisor Andrei Okounkov for his help and advice during this project. In particular, the crucial definition of the moduli of flags of sheaves of Section \ref{sec:fine} was the product of many discussions with him on this subject. I would also like to thank Boris Feigin and Eric Vasserot for patiently explaining their viewpoints on this subject. Many thanks are due to Roman Bezrukavnikov, Eugene Gorsky and Alexander Tsymbaliuk for their suggestions on the present text, and to Davesh Maulik for carefully editing a previous version of the title of this paper. \\

\section{The Moduli Space of Sheaves on $\BP^2$}
\label{sec:geom}

\subsection{}
\label{sub:sheaf}

Consider the projective plane, and fix a line $\infty \subset \BP^2$. Fix a number $r \in \BN$, and let $\CM = \CM(r)$ denote the moduli space of rank $r$ torsion free sheaves $\CF$ on $\BP^2$, together with an isomorphism (framing): 
$$
\CF|_\infty \cong \CO_\infty^{\oplus r}.
$$ 
This latter condition forces $c_1(\CF)=0$, but $c_2(\CF)$ is still free to range over the non-positive integers. For $d \geq 0$, we denote by $\CM_d \subset \CM$ the connected component of rank $r$ sheaves of second Chern class $-d \cdot [\text{pt}]$. Its tangent spaces are given by:
\begin{equation}
\label{eqn:ks}
T_{\CF} \CM_d = \textrm{Ext}^1(\CF, \CF(-\infty))
\end{equation}
by the Kodaira-Spencer isomorphism. Using this, one can easily prove that $\CM_d$ is smooth of dimension $2rd$. We have a universal sheaf $\CS$ on $\CM_d \times \BP^2$, and pushing it forward under the standard projection gives us the \textbf{tautological vector bundle}:
\begin{equation}
\label{eqn:taut}
\CT = R^1\textrm{pr}_{1*} (\CS (-\infty))
\end{equation}
on $\CM_d$. The above twist is by the pull-back of the divisor $\infty \subset \BP^2$, and it forces $R^0$ and $R^2$ to vanish. Therefore, $\CT$ is a vector bundle, and a standard application of the Riemann-Roch Theorem shows that it has rank $d$. \\

\subsection{}
\label{sub:adhm}

Consider the vector space:
$$
\textrm{End}(\BC^d) \oplus \textrm{End}(\BC^d) \oplus \textrm{Hom}(\BC^r,\BC^d)\oplus \textrm{Hom}(\BC^d,\BC^r) \stackrel{\mu}\longrightarrow \textrm{End}(\BC^d)
$$

\begin{equation}
\label{eqn:mom}
\mu(X,Y,A,B) = [X,Y]+AB
\end{equation}
The group $GL_d$ acts on this vector space by conjugating $X$ and $Y$, left multiplying $A$ and right multiplying $B$. The well-known ADHM description presents the moduli of sheaves as:
$$
\CM_d = \mu^{-1}(0)^s/GL_d
$$
where the superscript $s$ means "semistable points" and refers to the open set of points $(X,Y,A,B)\in \mu^{-1}(0)$ such that $\BC^d$ is generated by $X$ and $Y$ acting on $\im A$. A dimension count shows that:
$$
\dim(\CM_d) = d^2 + d^2 + dr + dr - d^2 - d^2 = 2dr
$$
where the first four terms come from the degrees of freedom in $X,Y,A,B$, while the last two terms come from the condition $\mu(X,Y,A,B)=0$ and from gauge transformations in $GL_d$. Note that the tautological vector bundle $\CT$ of \eqref{eqn:taut} simply has fibers $\BC^{d}$ in the ADHM description, but it is non-trivial on the whole of $\CM_d$ because we are taking the quotient by $GL_d$. \\

\subsection{}
\label{sub:toract}

The sheaf picture and the ADHM picture of $\CM_d$ are equivalent, and we refer the reader to \cite{Na} for the details. We will refer to a point of $\CM_d$ either as a sheaf $\CF$ or as a quadruple $(X,Y,A,B)$. We will henceforth fix coordinates $[x:y:z]$ on $\BP^2$, with respect to which $\infty$ is $\{z=0\}$. Then the 2-dimensional torus $\BC^* \times \BC^*$ acts on $\CM_d$: \\

\begin{itemize}

\item on sheaves $\CF$ by:
$$
(q_1,q_2)\cdot \CF = \phi^{q_1,q_2}_*\CF, \qquad \textrm{where } \ \phi^{q_1,q_2}(x,y) = (q_1^{-1}x,q_2^{-1}y)
$$

\item on quadruples $(X,Y,A,B)$ by:
$$
(q_1,q_2)\cdot (X,Y,A,B) = (q_1 X, q_2 Y, A, q_1q_2  B)
$$
\text{}

\end{itemize}

\begin{flushleft} and the group $GL_r$ acts on $\CM_d$: \\ \end{flushleft}

\text{}

\begin{itemize}

\item on sheaves $\CF$ by:
$$
g\cdot (\CF , \CF|_\infty \cong \CO_\infty^{\oplus r}) \longrightarrow (\CF , \CF|_\infty \cong \CO_\infty^{\oplus r} \stackrel{g^{-1}}\cong \CO_\infty^{\oplus r})
$$

\item on quadruples $(X,Y,A,B)$ by:
$$
g\cdot (X,Y,A,B) = (X,Y,Ag^{-1},gB)
$$
\textrm{}

\end{itemize}

\subsection{}
\label{sub:universal}

We consider the maximal torus $T  \subset GL_r \times \BC^* \times \BC^*$ acting on $\CM_d$ as in the previous section. We will study the $T-$equivariant derived categories of coherent sheaves:
$$
\CC_d = D^b_T(\coh(\CM_d))
$$
The tautological vector bundle $\CT$, in degree 0, is our first example of an object in these derived categories. Another interesting example is the complex:
\begin{equation}
\label{eqn:complex}
\CW = [q \CT  \stackrel{\Phi}\longrightarrow q_2\CT \oplus q_1\CT \oplus \CO^{\oplus r}  \stackrel{\Psi}\longrightarrow \CT]
\end{equation}
in degrees $-1,0,1$, where $q_1$ and $q_2$ are the elementary characters of $\BC^* \times \BC^*$ (interpreted as trivial line bundles with a non-trivial torus action) and $q=q_1q_2$. The maps in the above complex are given by:
\begin{equation}
\label{eqn:maps}
v \stackrel{\Phi}\longrightarrow (Xv,Yv,Bv), \qquad \qquad (v_1,v_2,w) \stackrel{\Psi}\longrightarrow Xv_2 - Y v_1 + A w
\end{equation}
The semistability condition forces $\Psi$ to be surjective, hence $\text{Ker }\Psi$ is a vector bundle and thus $\CW$ is quasi-isomorphic to a two-step complex of vector bundles. We will call $\CW$ the \textbf{universal complex}, and we do so because it is a resolution of the restriction of the universal sheaf $\CS$ to $\CM_d \times [0:0:1] \cong \CM_d$. \\ 





\subsection{} We will also consider the Grothendieck groups of the categories $\CC_d$, namely the equivariant $K-$theory groups:
$$
K_d = K^*_T(\CM_d)
$$
which are all modules over $\kk = \BC[t_1^{\pm 1},...,t_r^{\pm 1},q_1^{\pm 1},q_2^{\pm 1}]$. Here, $t_i$ are equivariant parameters of the maximal torus of $GL_r$, while $q_1$ and $q_2$ are equivariant parameters in the factors of $\BC^* \times \BC^*$. It will be convenient to work with all these spaces together:
$$
K = \bigoplus_{d \geq 0} K_d
$$
We have the following equality of classes for the complex of \eqref{eqn:complex}:
\begin{equation}
\label{eqn:gamma}
[\CW] = \sum_{k=1}^r t_k^{-1} - (1-q_1)(1-q_2) \cdot [\CT]
\end{equation}
For any vector bundle $\CV$ on $\CM$, we define its $\Lambda$ class as:
$$
\Lambda(\CV,u) = \sum_{i=0}^{\text{rk } \CV} (-u)^i [\Lambda^i \CV^\vee] \in K[u]
$$
We can extend this notion multiplicatively to any complex on $\CM$, in particular to $\CW$ of \eqref{eqn:complex}. The $\Lambda$ classes of the tautological bundle are particularly important because they generate the entire $K-$theory group. The following proposition will be proved in Subsection \ref{sub:gen}: \\

\begin{proposition}
\label{prop:gen}

For each $d\geq 0$, the vector space:
$$
K_d \otimes_{\kko} \emph{Frac}(\kko)
$$ 
is generated by products of tautological classes. \\

\end{proposition}

\section{The Moduli Space of Flags of Sheaves on $\BP^2$}
\label{sec:fine}

\subsection{} We have so far studied the moduli spaces $\CM_d$ separately. But we are interested in studying the relations between them, and this starts with the so-called \textbf{simple correspondences}:
$$
\fZ_{d,d+1} = \{\CF \supset_{p} \CF'\} \quad \subset \quad \CM_d \times \CM_{d+1}
$$
where $p = [0:0:1] \in \BP^2$. Here, the notation $\supset_p$ means that the quotient of the two sheaves is a skyscraper sheaf supported at $p$. These correspondences are known to be smooth of dimension $r(2d+1)-1$, although we will not need this. \\

\subsection{} If we generalized the above definition by requiring $\text{length}(\CF/\CF')=k$, we would obtain a correspondence which is too singular and too coarse for our needs. Instead, we will consider the \textbf{moduli of flags of sheaves}:
\begin{equation}
\label{eqn:flag}
\fZ_{d,d+k} = \{\CF_0 \supset_{p} \CF_1 \supset_{p} ... \supset_{p} \CF_{k}\} \ \subset \ \CM_d \times ... \times \CM_{d+k} 
\end{equation}
The following conjecture will be explained in the following sections, after we introduce the ADHM picture of $\fZ_{d,d+k}$. Note that while it would be nice to have a proof of this conjecture, we can get around it for the purposes of this paper. \\


\begin{conjecture}
\label{conj:struct}

The variety $\fZ_{d,d+k}$ is a local complete intersection of dimension 
$$(2d+k)r-1$$ 

\end{conjecture}

\begin{remark}
\label{rem:duh0}

Note that the expected dimension of $\fZ_{d,d+k}$ is $k-1$ greater than the expected dimension of the composed correspondence $\fZ_{d,d+1} \circ... \circ \fZ_{d+k-1,d+k}$. Therefore, the functor defined by $\fZ_{d,d+k}$ as a correspondence is not simply the composition of the usual Nakajima functors. We will come back to this point in Remark \ref{rem:duh}. \\

\end{remark}


\subsection{} 
\label{sub:flagadhm}

In the ADHM picture, a point of $\fZ_{d,d+k}$ is given by tuples of matrices $(X_i,Y_i,A_i,B_i) \in \CM_{i}$ for $i\in \{d,...,d+k\}$, which preserve a fixed flag of quotients:
\begin{equation}
\label{eqn:flaggy}
\BC^{d+k} \twoheadrightarrow \BC^{d+k-1} \twoheadrightarrow ... \twoheadrightarrow \BC^d
\end{equation}
The matrices $X_{d},Y_{d},...,X_{d+k-1},Y_{d+k-1}$ are all determined by $X=X_{d+k}$ and $Y=Y_{d+k}$, which lie in the subspace $\mat_{d,d+k}$ of order $d+k$ matrices that act nilpotently on the flag \eqref{eqn:flaggy}. In coordinates, such matrices take the following form (pictured below for $d=2$ and $k=3$):
\begin{equation}
\label{eqn:matrixpres}
\left( \begin{array}{ccccc}
0 & * & * & * & * \\
0 & 0 & * & * & * \\
0 & 0 & 0 & * & * \\
0 & 0 & 0 & * & * \\
0 & 0 & 0 & * & * \end{array} \right)
\end{equation}
Similarly, the maps $A_{d+i}$ are all determined by $A = A_{d+k}:\BC^r \longrightarrow \BC^{d+k}$ and the maps $B_{d+i}$ are all determined by $B = B_d: \BC^{d} \longrightarrow \BC^r$. If we let:
$$
\mat_{d,d+k} \oplus \mat_{d,d+k} \oplus \Hom(\BC^r,\BC^{d+k}) \oplus \Hom(\BC^d,\BC^r) \stackrel{\eta}\longrightarrow \mat'_{d,d+k} 
$$

$$
\eta(X,Y,A,B) = [X,Y]+AB \in \mat_{d,d+k}
$$
\footnote{Here, we let $\mat'_{d,d+k} \subset \mat_{d,d+k}$ denote the subspace of matrices such that the first $k-1$ entries on the superdiagonal vanish} we see that:
$$
\fZ_{d,d+k} = \eta^{-1}(0)^s/GL_{d,d+k}
$$
where $GL_{d,d+k}$ denotes the subgroup of invertible matrices which preserve the flag \eqref{eqn:flaggy}. As before, the superscript $s$ denotes semistable points, i.e. those such that the whole $\BC^{d+k}$ is generated by $X$ and $Y$ acting on $\im A$. \\

\subsection{} The dimension of the affine space of matrices $(X,Y,A,B)$ as above is $2d^2 + 2dk + k(k-1) + (2d+k)r$, whereas the number of equations imposed by setting $\eta=0$ is $d^2 + dk + \frac {(k-1)(k-2)}2$. This implies that $\eta^{-1}(0)$ is an affine variety of dimension:
\begin{equation}
\label{eqn:dimcount}
\geq d^2 + dk + \frac {(k-1)(k+2)}2 + (2d+k)r
\end{equation}
Since $\fZ_{d,d+k}$ is the quotient of $\eta^{-1}(0)^s$ by a free action of a group of dimension $d^2 + dk + \frac {k(k+1)}2$, we have:
$$
\text{dim}(\fZ_{d,d+k}) \geq (2d+k)r-1
$$
Conjecture \ref{conj:struct} claims that the above inequality is an equality. If we assume this conjecture, by the above discussion, the condition $\eta = 0$ gives a system of equations that (locally) cut out $\fZ_{d,d+k}$. \\


\subsection{} The variety $\fZ_{d,d+k}$ comes with natural projection maps to $\CM_d,...,\CM_{d+k}$, and we can talk about lifting tautological vector bundles from the various spaces $\CM_{d+i}$. Restricted to $\fZ_{d,d+k}$, these bundles form a flag:
\begin{equation}
\label{eqn:flagtaut}
\CT_{d+k} \twoheadrightarrow \CT_{d+k-1} \twoheadrightarrow ... \twoheadrightarrow \CT_{d}
\end{equation}
whose fibers are precisely the flag \eqref{eqn:flaggy}. Because the maps $X,Y:\CT_\bullet \longrightarrow \CT_\bullet$ are nilpotent, they can actually be extended uniquely to maps $\CT_\bullet \longrightarrow \CT_{\bullet+1}$. We will be very interested in:
$$
\CL_i = \Ker(\CT_{d+i} \twoheadrightarrow \CT_{d+i-1}) \qquad \forall i\in \{1,...,k\}
$$
which will be called \textbf{tautological line bundles} on $\fZ_{d,d+k}$. We will also consider the following universal complexes on $\fZ_{d,d+k}$, close counterparts of \eqref{eqn:complex}:
\begin{equation}
\label{eqn:w1}
\tCW_1 = [\CT_{d}  \stackrel{\Phi_1}\longrightarrow q_1^{-1} \CT_{d} \oplus q_2^{-1}\CT_{d} \oplus q^{-1}\CO^{\oplus r}  \stackrel{\Psi_1}\longrightarrow q^{-1}\CT_{d+1}]
\end{equation}

\begin{equation}
\label{eqn:w2}
\tCW_2 = [q\CT_{d+k-1}  \stackrel{\Phi_2}\longrightarrow q_2\CT_{d+k} \oplus q_1\CT_{d+k} \oplus \CO^{\oplus r}  \stackrel{\Psi_2}\longrightarrow \CT_{d+k}]
\end{equation}
where the maps $\Phi$ and $\Psi$ are given by the same formulas \eqref{eqn:maps}. Note that one of the tautological bundles in each of the above complexes has a different index from the others. \\  

\subsection{} The advantage in studying the spaces $\fZ_{d,d+k}$ is that we can describe them inductively in $k$. Namely, we have two projection maps:
\begin{equation}
\label{eqn:projmaps}
\xymatrix{& \fZ_{d,d+k} \ar[ld]_{\pi^-} \ar[rd]^{\pi^+} & \\
\fZ_{d,d+k-1} & & \fZ_{d+1,d+k}}
\end{equation}
given in the sheaf picture by forgetting the last (respectively, first) sheaf in the flag \eqref{eqn:flag}. In the ADHM picture, the projection maps are given by forgetting $\BC^{d+k}$ (respectively, $\BC^{d}$) in \eqref{eqn:flaggy}. \\

\begin{theorem}
\label{thm:towers}

If we assume Conjecture \ref{conj:struct}, then:
\begin{equation}
\label{eqn:ram1}
\pi^+: \fZ_{d-1,d+k} = \BP (\tCW_1^\vee[1]) \longrightarrow \fZ_{d,d+k}
\end{equation}
\begin{equation}
\label{eqn:ram2}
\pi^-: \fZ_{d,d+k+1} = \BP (\tCW_2) \longrightarrow \fZ_{d,d+k}
\end{equation}
\footnote{When $V$ is a sheaf, we write $\BP(V) = \text{Proj}(S^*V)$. When $V$ is a complex, the same notation holds, but the resulting object is a DG-scheme} The tautological line bundles which are forgotten by the above projection maps are connected with the Serre twisting sheaves via:
\begin{equation}
\label{eqn:line}
\CL_1 = \CO_{\pi^+}(-1), \qquad \CL_{k+1} = \CO_{\pi^-}(1)
\end{equation}

\textbf{} 

\end{theorem}


\begin{proof} To prove \eqref{eqn:ram1}, let us first study the fiber of $\pi^+$ above an arbitrary point $p \in \fZ_{d,d+k}$. The fiber consists of colength 1 quotients $\CT_d|_p \twoheadrightarrow \CT_{d-1}|_p$, together with maps $X,Y$ and $B$ that agree with the ones defining $p$. This datum is equivalent to a line inside $\CT_d$ which is annihilated by $X,Y$ and $B$, in other words, a line in: 
\begin{equation}
\label{eqn:jon}
\Ker \Phi_1|_p = H^{-1}(\tCW_1|_p)
\end{equation} 
However, it is not true that the fiber of $\pi^+$ above $p$ is the projectivization of the above vector space. That is because we have yet to impose the equations $\eta=0$, which by Conjecture \ref{conj:struct} cut out the varieties $\fZ_{d,d+k}$ and $\fZ_{d-1,d+k}$ scheme-theoretically. The $-1$-st and $0$-th terms of the complex $\tCW^1$ are precisely local generators and equations that cut out $\fZ_{d-1,d+k}$. Therefore, $\fZ_{d-1,d+k}$ is the projectivization of the whole complex $\tCW^\vee_1[1]$, precisely the claim \eqref{eqn:ram1}.
 
\textbf{} \\
As for \eqref{eqn:ram2}, let us first show that the fiber of $\pi^-$ above an arbitrary point $p\in \fZ_{d,d+k}$ is given by the following $r$-dimensional projective space:
\begin{equation}
\label{eqn:snow}
(\pi^-)^{-1}(p) = \BP\left(H^0(\tCW_2|_p) \right)
\end{equation}
Indeed, the fiber of $\pi^-$ above $p$ consists of all colength 1 extensions $\CT_{d+k+1}|_p \twoheadrightarrow \CT_{d+k}|_p$, together with commuting maps $X,Y$ and $A$. Since $X$ and $Y$ are nilpotent, this comes down to a one-dimensional extension:
$$
\xymatrix{q_2 \CT_{d+k}|_p \oplus q_1 \CT_{d+k}|_p \oplus \CO|_p^{\oplus r} \ar@{-->}[rd] \ar[r]^-{\Psi_2} & \CT_{d+k}|_p \\
& \CT_{d+k+1}|_p \ar@{->>}[u]}
$$
or in other words, a line in the cokernel of $(\Psi_2|_p)^\vee$. Since the extended maps $X,Y$ and $A$ must satisfy the moment map equations $[X,Y]+AB=0$, this line must also lie in the kernel of $(\Phi_2|_p)^\vee$. This implies that the fiber of $\pi^-$ above $p$ is the projective space of lines inside:
$$
\text{Ker}(\Phi_2|_p)^\vee / \ \text{Im}(\Psi_2|_p)^\vee = H^0(\tCW_2|^\vee_p)
$$
This proves \eqref{eqn:snow}. To obtain \eqref{eqn:ram2}, one needs to show that the complex $\tCW_2$ has no other cohomology than in degree 0. Semistability forces $\Psi_2$ to be surjective, and hence the complex $\tCW_2$ is exact at the last step. Meanwhile, if it failed to be exact at the first step, we would have an open locus of points $p\in \fZ_{d,d+k}$ where $\Phi_2$ drops rank. This would imply an open locus of points $p\in \fZ_{d,d+k}$ where $H^0(\tCW_2|_p)$ has dimension larger than expected. This would imply that $\fZ_{d,d+k+1}$ would have dimension larger than expected, contradicting Conjecture \ref{conj:struct}. 



\end{proof}

\subsection{} Since we don't yet have a proof of Conjecture \ref{conj:struct}, we let Theorem \ref{thm:towers} be the \textbf{definition} of the moduli spaces $\fZ_{d,d+k}$. By this we mean that we define:
$$
\fZ^\pm_{d,d} = \CM_d
$$
and then iteratively define the DG-schemes $\fZ^\pm_{d,d+k}$ as the projective towers given by the complexes \eqref{eqn:ram1} and \eqref{eqn:ram2}.  Then Theorem \ref{thm:towers} claims that Conjecture \ref{conj:struct} implies $\fZ^\pm_{d,d+k} \cong \fZ_{d,d+k}$. 

\textbf{} \\
With this nuisance in mind, let us note that all the functors defined in this paper will (strictly speaking) pertain to the DG-schemes $\fZ^\pm_{d,d+k}$. However, Conjecture \ref{conj:struct} is supported by a significant amount of evidence, and so we will suppress the extra notation $\pm$. Henceforth, in all our constructions, we will label either of these spaces simply by $\fZ_{d,d+k}$, and keep the derived picture only in the back of our minds. \\ 

\subsection{} The moduli of flags acts as a correspondence between $\CM_d$ and $\CM_{d+k}$:
\begin{equation}
\label{eqn:projmaps2}
\xymatrix{& \fZ_{d,d+k} \ar[ld]_{p^-} \ar[rd]^{p^+} & \\
\CM_d & & \CM_{d+k}}
\end{equation}
Indeed, one can use the maps \eqref{eqn:projmaps2} to define mutually adjoint functors on the equivariant derived categories of the moduli spaces $\CM_\bullet$:
$$
\CC_\bullet \longrightarrow \CC_{\bullet \pm k}, 
$$
\begin{equation}
\label{eqn:derived}
\qquad c \longrightarrow Rp^\pm_*(m(\CL_1,...,\CL_k) \otimes p^{\mp *}(c))
\end{equation}
for any Laurent polynomial $m \in \kk[z^{\pm 1}_1,...,z_k^{\pm 1}]$. If we do not wish to assume Conjecture \ref{conj:struct}, then we need to define the push-forward maps $Rp^\pm_*$ as in the previous Subsection: a composition of push-forwards down a projective tower. We will study the linear maps induced by these functors at the level of $K-$theory:
$$
x^\pm_m : K_\bullet \longrightarrow K_{\bullet \pm k} 
$$
\begin{equation}
\label{eqn:opop}
x^\pm_m(c) = p^\pm_*\left[ m(l_1,...,l_k) \cdot p^{\mp *}(c) \right], \qquad \forall c\in K_\bullet
\end{equation}
where $l_i = [\CL_i] \in K^*_T(\fZ_{\bullet, \bullet+k})$. In Theorem \ref{thm:main}, we will see that the maps \eqref{eqn:opop} are described by certain elements in the shuffle algebra, to be defined in Section \ref{sec:shuf}. \\

\subsection{} In order to compute the linear maps \eqref{eqn:opop}, we will need to know how to push-forward classes under the projectivizations of Theorem \ref{thm:towers}:
$$
\pi^+:\fZ_{d-1,d+k} \longrightarrow \fZ_{d,d+k}, \qquad \qquad \pi^-:\fZ_{d,d+k+1} \longrightarrow \fZ_{d,d+k}
$$ 
We let $l$ denote the class of the tautological line bundle which is forgotten by the maps $\pi^+$ (resp. $\pi^-$), i.e. $l=l_1$ (resp. $l=l_k$). We have the following result: \\

\begin{lemma}
\label{lem:push}

For any rational function $r(u)$ with coefficients in $\fZ_{d,d+k}$, we have:
$$
\pi^+_*(r(l)) = \int r(u) \Lambda(\tCW_1,u) Du
$$

$$
\pi^-_*(r(l)) = - \int r(u) \Lambda(-\tCW^\vee_2,u^{-1}) Du
$$
where the contour separates the following subsets of complex numbers:
$$
\emph{Poles}(r(u)) \cup \{0,\infty\} \qquad \emph{from} \qquad \emph{Poles}(\Lambda(\pm \tCW^*_\flat, u^{\pm 1}))
$$ 
\footnote{where $(*,\flat) = ( \text{ } , 1)$ or $(\vee,2)$} For a rational function with coefficients in $K-$theory, we define its poles by formally decomposing $K-$theory classes into Chern roots, and treating these Chern roots as complex numbers. We abbreviate $Du = \frac {du}{2\pi i u}$. \\

\end{lemma}


\begin{proof} Because of the choice of contour, it is enough to prove the statement for $r(u)=u^k$. In this case, \eqref{eqn:line} gives us:
$$
\pi^+_*(l^k) = \pi^+_*[\CO(-k)] = \left[ S^{-k}\tCW_1^\vee[-1] \right] 
$$
Let us write $[\tCW_1] = \sum \pm w_i$ formally as an alternating sum of Chern roots, and obtain:
$$
\pi^+_*(l^k) = \text{coefficient of }u^{-k}\text{ in }\prod_i \left(1 - \frac u{w_i}\right)^{\pm 1} = \int u^k \Lambda(\tCW_1,u) Du
$$
In Section \ref{sec:fixed}, we will interpret the Chern roots as the virtual characters of the torus action in the fibers of $\tCW_1$ above fixed points. The case of $\pi^-$ is treated analogously. 


\end{proof}

\section{The Double Shuffle Algebra}
\label{sec:shuf}

\subsection{}

Consider an infinite set of variables $z_1,z_2,...$, and take the $\kk-$vector space:
\begin{equation}
\label{eqn:big}
V = \bigoplus_{k \geq 0} \kk(z_{1},...,z_{k})^{\sym}
\end{equation}
of rational functions which are symmetric in the variables $z_1,...,z_k$. We can endow this vector space with a $\kk-$algebra structure by the multiplication:
$$
P(z_{1},...,z_{k}) * Q(z_{1},...,z_{l}) =
$$

\begin{equation}
\label{eqn:mult}
= \frac 1{k! l!} \cdot \textrm{Sym} \left[P(z_{1},...,z_{k})Q(z_{k+1},...,z_{k+l}) \prod_{i=1}^k \prod_{k+1 = j}^{k+l} \omega \left( \frac {z_i}{z_j} \right) \right]
\end{equation}
where:
\begin{equation}
\label{eqn:factor}
\omega(x)  = \frac {(x - 1)(x - q)}{(x - q_1)(x - q_2)} , \qquad \qquad q=q_1q_2
\end{equation}
and \textrm{Sym} denotes the symmetrization operator:
$$
\textrm{Sym}\left( P(z_1,...,z_k) \right) = \sum_{\sigma \in S(k)} P(z_{\sigma(1)},...,z_{\sigma(k)})
$$

\subsection{}

The \textbf{shuffle algebra} $\CA^+$ is defined as the subspace of $V$ consisting of rational functions of the form:
\begin{equation}
\label{eqn:shuf}
P(z_{1},...,z_{k}) = \frac {p(z_{1},...,z_{k}) \cdot \prod_{i \neq j} (z_{i} - z_{j})}{\prod_{i \neq j} (z_{i} - q_1 \cdot z_{ j})(z_{i} - q_2 \cdot z_{j})} 
\end{equation}
where $p$ is a symmetric Laurent polynomial that satisfies the \textbf{wheel conditions}:
\begin{equation}
\label{eqn:wheel}
p(z,q_1  z,q_1q_2  z,w_4,...,w_{k}) = p(z,q_2  z,q_1q_2  z,w_4,...,w_{k}) = 0
\end{equation}
for all variables $z,w_4,...,w_k$. This condition is vacuous for $k\leq 2$. It is straightforward to show that $\CA^+$ is an algebra, and it is shown in \cite{Shuf} to be generated by the rational functions in one variable $z^d :=z_1^d$, as $d$ goes over $\BZ$. \\




\subsection{}
\label{sub:cartan}

Let us define the \textbf{extended shuffle algebra} $\CA^\geq$ to be generated by $\CA^+$ and commuting generators $H_0,H_1,...$, under the relation:
\begin{equation}
\label{eqn:comm0}
P(z_{1},...,z_{k}) * H(z) = H(z) * \left[ P(z_{1},...,z_{ k}) \prod_{i=1}^k  \frac {\omega(z_i/z)}{\omega(z/z_i)}  \right]
\end{equation}
where:
$$
H(z) =  \sum_{n \geq 0} H_{n} \cdot z^{- n}
$$
Note that $H_0$ is central, and define commuting generators $p_1,p_2,...$ by:
\begin{equation}
\label{eqn:defp}
H(z) = H_0 \cdot \exp \left(\sum_{n=1}^\infty p_n \cdot \frac { (q_1^n-1)(q_2^n-1)(q^{-n}-1)z^{-n}}n\right)
\end{equation}
Note that relation \eqref{eqn:comm0} is equivalent to:
\begin{equation}
\label{eqn:comm1}
[p_n, P(z_{1},...,z_{k})] = P(z_{1},...,z_{k}) (z_1^n+...+z_k^n)
\end{equation}






\subsection{} It was shown in \cite{Shuf} that $\CA^\geq$ possesses a coproduct and a bialgebra pairing\footnote{The coproduct and pairing depend in \loccit on two parameters $c_1$ and $c_2$, which in the present paper are set equal to $c_1=q^{-1}$ and $c_2=1$}, which allow us to construct the Drinfeld double $\CA = \DD \CA^\geq$ (see \cite{Shuf} for details). This double algebra is isomorphic to the elliptic Hall algebra, which in turn is isomorphic to the spherical part of the ($N\longrightarrow \infty$) double affine Hecke algebra of type $\fsl_N$. Explicitly, the algebra $\CA$ has generators: 
$$
\CA = \left \langle P^+, \widetilde{P}^-, H_n^+, H_n^- \right \rangle 
$$
over all shuffle elements $P\in \CA^+$ and over all $n\geq 0$. As algebras, the negative generators $\widetilde{P}^-$ satisfy the same relations as those satisfied by the positive generators $P^+$. We will actually employ the slightly different notation:
$$
P^- = \widetilde{\tau(P)}^-
$$
where $\tau:\CA^+ \longrightarrow \CA^+$ is the anti-automorphism $P(z_1,...,z_k) \longrightarrow P(z_1^{-1},...,z_k^{-1})$. Now, the $P^-$ satisfy the \textbf{opposite} relations as those satisfied by the $P^+$:
$$
(P*Q)^- = Q^- * P^-
$$
for all $P,Q \in \CA^+$. \\ 

\subsection{} As for the relations between positive and negative generators, they look simpler when defined on the degree 1 generators:

\begin{equation}
\label{eqn:drinfeld}
[(z^d)^+ , (z^{d'})^-] = \frac {(1-q_1)(1-q_2)}{1-q} \left(H^-_{-d-d'}  \delta_{-d-d' \geq 0} - H^+_{d+d'}  \delta_{d+d'\geq 0}  \right)
\end{equation}
In fact, since the $z^d$ generate the whole algebra $\CA^+$, the above relation actually determines all relations between positive and negative shuffle algebra elements. \\

\subsection{}

For any set of variables $S$, we write: 
\begin{equation}
\label{eqn:lambda}
\Lambda_d^S = \prod_{s\in S} \Lambda(\CT_d,s),
\end{equation}
and note by Proposition \ref{prop:gen} that such classes span $K_d$. The following theorem was proved (in a different, but equivalent language), in \cite{FT} and \cite{SV}. Let $\e=1$ or $0$ depending on whether the sign is $\pm$, and also write $\CW^+ = \CW, \CW^- = \CW^\vee$: \\

\begin{theorem}
\label{thm:act}




There is an action of $\CA$ on $K$, given by $H^\pm = (-q^\e)^rt_1...t_r$,
\begin{equation}
\label{eqn:diagonal}
p^\pm_n\cdot c = \pm ([\CT]^{\pm n}) \cdot c
\end{equation}
\footnote{The exponent above $[\CT]$ refers to a plethysm: if $[\CT] = a_1+...+a_d$, then $[\CT]^n = a_1^n + ... + a_d^n$} and:
\begin{equation}
\label{eqn:formula}
P^\pm \cdot \Lambda_d^S = (\pm 1)^k \Lambda_{d\pm k}^S :\int:  P(u_1,...,u_k) \prod_{i=1}^k \frac {\Lambda(\pm \CW^\pm_{d \pm k}, u^{\pm 1}_iq^\e)}{u_i^{r\e}\prod_{s \in S} \left(1 - \frac s{u_i} \right)^{\pm 1}} \qquad
\end{equation}
for all $P\in \CA^+$. 

\end{theorem}

\textbf{} \\
To see that Theorem \ref{thm:act} corresponds to the ones in \cite{FT} and \cite{SV}, it is enough to check that their action of the degree one elements $(z^d)^\pm \in \CA^\pm$ on $K$ is given by the above integral formulas. In either of \loccit these elements are defined by tautological line bundles on the correspondences $\fZ_{d,d+1}$, which will be shown to equal \eqref{eqn:formula} in Theorem \ref{thm:main} below. Alternatively, one could prove Theorem \ref{thm:act} directly from the above formulas, by showing that they satisfy the relations in the shuffle algebra and \eqref{eqn:drinfeld}. This is a straightforward computation, and we leave it to the interested reader. \\

\subsection{} 
\label{sub:normal}

We define the above normal ordered integrals to only count the residues at $k-$tuples $(u_1,...,u_k)$ such that:

$$
\begin{cases}
u_i \in \textrm{Poles}(\Lambda(\pm \CW^{\pm},z^{\pm 1} q^\e)) \textrm{ or} \\
u_i = q_1 u_j \textrm{ for some } \pm j > \pm i \textrm{ or } \\
u_i = q_2 u_j \textrm{ for some } \pm j > \pm i 
\end{cases}
$$
To explain this way of counting residues, let us recall Proposition 3.5 of \cite{Shuf}, which states that any shuffle element is a linear combination of shuffle elements of the form:
\begin{equation}
\label{eqn:doth}
z^{n_1} * ... * z^{n_k} = \sym \left[z_1^{n_1} ... z_k^{n_k} \prod_{1\leq i < j \leq k} \omega \left( \frac {z_i}{z_j} \right) \right]
\end{equation}
For such shuffle elements, the above residue count gives rise to the following actual integral for the RHS of \eqref{eqn:formula}:
\begin{equation}
\label{eqn:lala}
(\pm 1)^k \Lambda^S_{d\pm k} \int u_1^{n_1} ... u_k^{n_k} \prod_{1\leq i < j \leq k} \omega \left( \frac {u_i}{u_j} \right) \prod_{i=1}^k \frac {\Lambda(\pm \CW_{d \pm k}^{\pm}, u^{\pm 1}_iq^{\e})}{u_i^{r\e} \prod_{s \in S} \left(1 - \frac s{u_i} \right)^{\pm 1}} Du_i
\end{equation}
where recall that $Du = du/2\pi iu$. The integrals are over contours that separate $S \cup\{0,\infty\}$ from $\text{Poles}(\Lambda(\pm \CW^{\pm},z^{\pm 1}q^\e))$, with $u_1$ being closest/farthest from the latter set depending on whether the sign is $+$ or $-$. It is easy to see that these formulas are actually imposed by iterating \eqref{eqn:formula} $k$ times to compute the action of $z^{n_1} * ... * z^{n_k}$. \footnote{The reason the order of the contours differs in the cases $+$ and $-$ is that the creation shuffle elements $P^+$ satisfy the opposite algebra relations from the annihilation shuffle elements $P^-$} \\



\subsection{} The shuffle algebra formalism allows us write down explicitly many elements of $\CA$, and then the above formulas tell us how they act on $K$. In particular, an important class of elements of $\CA$ that were defined in \cite{Shuf} are:
$$
X_{m} = \sym \left[ \frac {m(z_1,...,z_k)}{\left(1-\frac {z_2q}{z_1} \right) ... \left(1-\frac {z_{k}q}{z_{k-1}} \right)} \prod_{1\leq i<j\leq k} \omega \left( \frac {z_i}{z_j} \right) \right] \in \CA^+
$$
for any Laurent polynomial $m(z_1,...,z_k)$. It is shown in Proposition 6.2 of \cite{Shuf} that $X_{m} \in \CA^+$ for any $m$, and they therefore act on $K$ via Theorem \ref{thm:act} above. In fact, it is easy to see that \eqref{eqn:formula} implies:
\begin{equation}
\label{eqn:formulam}
X^\pm_m \cdot \Lambda_d^S = (\pm 1)^k \Lambda_{d\pm k}^S \int \frac {m(u_1,...,u_k) \prod_{i<j} \omega \left( \frac {u_i}{u_j} \right)}{\left(1-\frac {u_2q}{u_1} \right) ... \left(1-\frac {u_{k}q}{u_{k-1}} \right)}   \prod_{i=1}^k \frac {\Lambda(\CW_{d \pm k}, u^{\pm 1}_i q^\e)}{\prod_{s \in S} \left(1 - \frac s{u_i} \right)^{\pm 1}} Du_i \qquad 
\end{equation}
where the contours separate the sets $S \cup \{0,\infty\}$ from $\pol(\Lambda(\pm \CW^{\pm},z^{\pm 1} q^\e))$, with $u_1$ the contour closest/farthest from the latter set depending on whether the sign is $+$ or $-$. One of the main points of this paper is to give a geometric description of the operator \eqref{eqn:formulam}, and this will be achieved via the moduli space $\fZ^k$ of flags of sheaves: \\

\begin{theorem}
\label{thm:main}

For any Laurent polynomial $m(z_1,...,z_k)$, the geometric correspondence $x^\pm_{m}:K \longrightarrow K$ of \eqref{eqn:opop} is given by the shuffle element:

$$
X^\pm_{m(z_k,...,z_1)\cdot (z_1...z_k)^{r\e}} \in \CA^\pm \longrightarrow \emph{End}(K)
$$ 

\text{}

\end{theorem}

\begin{remark}
\label{rem:duh}
Theorem \ref{thm:main} shows us why the geometric correspondences $x_m$ are not simply compositions of simple Nakajima correspondences. If they were, the operators they define on $K-$theory would be shuffle elements of the form \eqref{eqn:doth}, and thus act on $K$ by \eqref{eqn:lala}. These differ from the shuffle elements $X_m$ by the absence of the denominator:
$$
\left(1-\frac {z_2q}{z_1} \right) ... \left(1-\frac {z_{k}q}{z_{k-1}} \right)
$$
In fact, this denominator appears in the shuffle element associated to $x_m^\pm$ because of $\CT_{d+1}$ (respectively $\CT_{d+k-1}$) in the definition of the complexes \eqref{eqn:w1} (respectively \eqref{eqn:w2}) that define the correspondences $x^+_m$ (respectively $x^-_m$). \\
\end{remark}

\begin{proof} We need to compare the geometric operators $x^\pm_m$ with the shuffle elements $X^\pm_m$. By Proposition \ref{prop:gen}, it will be enough to compare their action on products of tautological classes. By \eqref{eqn:opop}, we have:
$$
x^\pm_m \cdot \Lambda^S_d = p^{\pm}_* \left[ m(l_1,...,l_k)\prod_{s\in S} \Lambda(p^{\mp *}(\CT_d),s) \right]
$$
Comparing the various tautological sheaves on the variety $\fZ_{d,d+k}$, we see that:
\begin{equation}
\label{eqn:bb}
\L(p^{\mp *}(\CT_d),s) = \L(p^{\pm *}(\CT_{d\pm k}),s) \prod_{i=1}^k \left(1-\frac s{l_i} \right)^{\mp 1}
\end{equation}
and therefore:
\begin{equation}
\label{eqn:1}
x^\pm_m \cdot \Lambda_d^S = \Lambda_{d\pm k}^S \cdot p^{\pm}_* \left[ m(l_1,...,l_k)\prod_{s\in S} \prod_{i=1}^k \left(1-\frac s{l_i} \right)^{\mp 1} \right] \qquad
\end{equation}
The maps $p^\pm$ factor as:
$$
p^+ : \fZ_{d,d+k} \stackrel{\pi^+_{1}}\longrightarrow \fZ_{d+1,d+k} \stackrel{\pi^+_{2}}\longrightarrow ...  \stackrel{\pi^+_{k-1}}\longrightarrow \fZ_{d+k-1,d+k} \stackrel{\pi^+_{k}}\longrightarrow  \CM_{d+k}
$$

$$
p^- : \fZ_{d-k,d} \stackrel{\pi^-_{k}}\longrightarrow \fZ_{d-k,d-1} \stackrel{\pi^-_{k-1}}\longrightarrow ...  \stackrel{\pi^-_{2}}\longrightarrow \fZ_{d-k,d-k+1} \stackrel{\pi^-_{1}}\longrightarrow  \CM_{d-k}
$$
where the individual maps $\pi_i^\pm$ are the projectivizations that appear in Theorem \ref{thm:towers}. By Lemma \ref{lem:push}, we have for any rational function $r$:
$$
\pi^+_{i*}(r(l_i)) = \int r(u) \Lambda(\tCW_1,u) Du =  \int \frac {r(u) \Lambda(\CW_{d+i},uq)}{\left(1- \frac {uq}{l_{i+1}}\right)}  Du
$$

$$
\pi^-_{i*}(r(l_i)) = - \int r(u) \Lambda(-\tCW^\vee_2, u^{-1}) Du = - \int \frac {r(u) \Lambda(-\CW^\vee_{d+i-k-1},u^{-1})}{\left(1- \frac {l_{i-1}q}{u}\right)} Du
$$
The equalities on the right simply follow from the definitions of $\tCW_1$ and $\tCW_2$, where $\CW_i$ denotes the universal complex \eqref{eqn:complex} on each $\CM_i$. We can use \eqref{eqn:bb} to express these complexes in terms of $\CW_{d+k}$ and $\CW_{d-k}$, respectively:
$$
\pi^+_{i*}(r(l_i)) =  \int \frac {r(u) \Lambda(\CW_{d + k},uq)}{\left(1- \frac {uq}{l_{i+1}}\right)} \prod_{i<j} \omega\left(\frac {l_j}u \right) Du
$$

$$
\pi^-_{i*}(r(l_i)) = - \int \frac {r(u) \Lambda(-\CW^\vee_{d - k},u^{-1})}{\left(1- \frac {l_{i-1}q}{u}\right)} \prod_{j<i} \omega\left(\frac u{l_j} \right) Du
$$
Iterating these integral formulas gives us the following equality for \eqref{eqn:1}:
$$
x^\pm_m \cdot \Lambda_d^S = (\pm 1)^k \Lambda_{d\pm k}^S \cdot \int  \frac {m(u_1,...,u_k) \prod_{i<j}  \omega\left(\frac {u_j}{u_i} \right) }{\left(1-\frac {u_1q}{u_2}\right)...\left(1-\frac {u_{k-1}q}{u_k} \right)} \prod_{i=1}^k \frac {\Lambda (\pm \CW_{d\pm k}^{\pm}, u^{\pm 1} q^{\e})}{\prod_{s\in S} \left(1-\frac s{u_i} \right)^{\pm 1}} Du_i
$$
where the contours separate $S \cup \{0,\infty\}$ from $\pol(\Lambda(\pm \CW,\cdot))$, with $u_1$ farthest/closest to the latter set, depending on whether the sign is $+$ or $-$. Changing the variables to $v_i = u_{k+1-i}$ gives us precisely \eqref{eqn:formulam}. 

\end{proof}

\subsection{} In \cite{Shuf}, we define explicit shuffle elements $P_{k,d} \in \CA$ for all $k,d \in \BZ^2 \backslash (0,0)$. It is shown in \loccit that they are the images of the elliptic Hall algebra generators studied in \cite{BS}, and are permuted by an (almost) $\slz$ action of automorphisms on $\CA$. To write them out explicitly, let us write $n = \gcd(k,d)$ and $a = \frac kn$. Then Theorem 1.1 of \cite{Shuf} states that \footnote{Note that our $P_{k,d}$ and those of \loccit are off by a constant of $\frac {(q_1-1)^k(1-q_2)^k}{(q_1^{n}-1)(1-q_2^{n})}$. This is done in the present paper purely for cosmetic reasons}:
\begin{equation}
\label{eqn:heis}
P_{\pm k, d} =  X^\pm_{m_{k, d}}
\end{equation}
for all $k>0$, where the Laurent polynomial $m_{k, d}$ is given by:
$$
m_{k, d}(z_1,...,z_k) = {\prod_{i=1}^k z_i^{\left \lfloor \frac {id}k \right \rfloor - \left \lfloor \frac {(i-1)d}k \right \rfloor}\sum_{x=0}^{n-1} q^{x} \frac {z_{a(n-1)+1}...z_{a(n-x)+1}}{{z_{a(n-1)} ...z_{a(n-x)}}}}
$$
When $k=0$, the corresponding elements coincide with the Cartan generators of Subsection \ref{sub:cartan}: $P_{0,\pm d} = p^{\pm}_d$ for any $d>0$. \\



\subsection{} 
\label{sub:bara}

For each pair of coprime natural numbers $\gcd(a,b) = 1$, the elements $P_{na,nb}$ determine a $q-$Heisenberg algebra:
$$
[P_{na,nb}, P_{n'a,n'b}] =  \frac {\delta_{n+n'}^0   q^{-na}(q_1^n-1)(q^n_2-1)}{(q^{-n}-1)} \left[ (H_0^-)^{na} - (H_0^+)^{na} \right]
$$
for all $n>0$ and all $n'\in \BZ$. In particular, $\{P_{n,0}\}_{n\in \BZ}$ will determine a geometric action of the $q-$Heisenberg algebra on $K$ that deforms the construction of Baranovsky in cohomology. This follows from Theorem \ref{thm:main}, which implies that the $K-$theoretic version of Baranovsky's operators are:
\begin{equation}
\label{eqn:bar}
P_{\pm k,0}(c) = p^\pm_*\left[(l_1...l_k)^{-r\e} \sum_{i=1}^{k} q^{i-1} \frac {l_1}{l_{i}} \cdot p^{\mp *}(c)\right], \ \forall c\in K
\end{equation}
for all $k>0$, where $p^\pm:\fZ^k \longrightarrow \CM$ are the projections \eqref{eqn:projmaps2} from the moduli of flags of sheaves $\fZ^k = \sqcup_d \fZ_{d,d+k}$. A similar geometric result holds for all $P_{\pm k, d}$, simply by replacing the sum of monomials in the $l_i$ by $m_{k,d}(l_k,...,l_1)\cdot (l_1...l_k)^{-r\e}$.  \\

\subsection{} At the suggestion of Boris Feigin, we will prove the following result: \\

\begin{proposition}
\label{prop:feigin} The vector $v = \sum_{d\geq 0} 1_{d} \in K$ is an eigenvector of all $P_{k,d}$ corresponding to lattice points $(k,d)$ in a certain cone:
\begin{equation}
\label{eqn:boris}
P_{-k,0} \cdot v = v,
\end{equation}

\begin{equation}
\label{eqn:borya}
P_{-k,d} \cdot v = 0,
\end{equation}
for all $-kr < d < 0$, and:
\begin{equation}
\label{eqn:bor}
P_{-k,-kr} \cdot v = \left[ (-1)^{r} t_1...t_r\right]^{-k} \cdot v
\end{equation}
\textbf{} \\
\end{proposition}

\begin{proof} Formulas \eqref{eqn:formulam} and \eqref{eqn:heis} imply that $P_{-k,0} \cdot 1_d$ equals:
$$
(-1)^k \int  \frac {1 + \frac {u_kq}{u_{k-1}}+...+\frac {u_kq^{k-1}}{u_1}}{\left(1-\frac {u_2q}{u_1} \right) ... \left(1-\frac {u_{k}q}{u_{k-1}} \right)} \prod_{1\leq i<j\leq k} \omega \left( \frac {u_i}{u_j} \right) \prod_{i=1}^k \Lambda(- \CW^\vee_{d - k}, u_i^{-1}) Du_i
$$
where the contours go around the poles of the rational function $\Lambda(-\CW^\vee,z^{-1})$, with $u_1$ being the outermost one. We can deform the contours to small loops around $0$ and $\infty$, with $u_1$ being the closest to the these poles. There is no pole at 0 in $u_1$ (because the function $\Lambda(-\CW^\vee,z^{-1})$ vanishes to order $r$), but there is a simple pole at $\infty$ in $u_1$, because of the first fraction. The residue equals:
$$
(-1)^{k-1} \int  \frac {1 + \frac {u_kq}{u_{k-1}}+...+\frac {u_kq^{k-2}}{u_2}}{\left(1-\frac {u_3q}{u_2} \right) ... \left(1-\frac {u_{k}q}{u_{k-1}} \right)} \prod_{2\leq i<j\leq k} \omega \left( \frac {u_i}{u_j} \right) \prod_{i=1}^k \Lambda(- \CW^\vee_{d - k}, u_i^{-1}) Du_i
$$
We can integrate over $u_2,u_3,...$ in the same way, and the result yields \eqref{eqn:boris}.

\textbf{} \\
For any $-kr<d<0$, formula \eqref{eqn:formulam} implies that $P_{-k,d} \cdot 1_d$ equals:
$$
(-1)^k \int \frac {m_{k,d}(u_1,...,u_k)}{\left(1-\frac {u_2q}{u_1} \right) ... \left(1-\frac {u_{k}q}{u_{k-1}} \right)} \prod_{1\leq i<j\leq k} \omega \left( \frac {u_i}{u_j} \right) \prod_{i=1}^k    \Lambda(- \CW^\vee_{d - k}, u_i^{-1}) Du_i
$$
where by \eqref{eqn:heis} and $-kr<d<0$, the monomial $m_{k,d}$ only has terms of degree $\{-1,...,-r\}$ in $u_1$. Hence the above integral is regular at $u_1 = 0$, and it has at most a single pole at $u_1=\infty$, but only if it has degree exactly $r$ in $m_{k,d}$. If this happens, the corresponding residue is:
$$
(-1)^{k-1} \int \frac {m'(u_2,...,u_k)}{\left(1-\frac {u_3q}{u_2} \right) ... \left(1-\frac {u_{k}q}{u_{k-1}} \right)} \prod_{2\leq i<j\leq k} \omega \left( \frac {u_i}{u_j} \right) \prod_{2=1}^k    \Lambda(- \CW^\vee_{d - k}, u_i^{-1}) Du_i
$$
where the monomial $m'$ only has terms of degree $\{-1,...,-r\}$ in $u_2$. We may repeat the same argument, and this procedure will eventually end with 0, because the hypothesis $-kr<d$ ensures that at some step, the monomial in the numerator will not have degree $-r$ in the variable that needs integrating.

\textbf{} \\
Formulas \eqref{eqn:formulam} and \eqref{eqn:heis} imply that $P_{-k,-kr}\cdot 1_d$ equals:
$$
(-1)^k \int  \frac {1 + \frac {u_kq}{u_{k-1}}+...+\frac {u_kq^{k-1}}{u_1}}{\left(1-\frac {u_2q}{u_1} \right) ... \left(1-\frac {u_{k}q}{u_{k-1}} \right)} \prod_{1\leq i<j\leq k} \omega \left( \frac {u_i}{u_j} \right) \prod_{i=1}^n   u_i^{-r} \Lambda(- \CW^\vee_{d - k}, u_i^{-1}) Du_i
$$
where the contours go around the poles of the rational function $\Lambda(- \CW^\vee,z^{-1})$, with $u_1$ being the outermost one. We can deform the contours to small loops around $0$ and $\infty$, with $u_1$ being the closest to the these poles. There is no residue at $\infty$ in $u_1$, because of $u_1^{-r}$. As for the residue at $0$ in $u_1$, it equals:
$$
(-1)^{k-1} \int  \frac {(-1)^{r}t^{-1}_1...t^{-1}_r \cdot \frac{u_kq^{k-2}}{u_2}}{\left(1-\frac {u_3q}{u_2} \right) ... \left(1-\frac {u_{k}q}{u_{k-1}} \right)} \prod_{2\leq i<j\leq k} \omega \left( \frac {u_i}{u_j} \right) \prod_{i=2}^{k}   u_i^{-r} \Lambda(- \CW_{d - k}^\vee, u_i^{-1}) Du_i
$$
We can now integrate over $u_2,u_3,...$ in the same way, and the result yields \eqref{eqn:bor}.

\end{proof}

\section{Fixed points}
\label{sec:fixed}

\subsection{} 
\label{sub:part}

In this section, we will use the language of partitions $\lambda = (\lambda_0 \geq \lambda_1 \geq ...)$. To any such partition, we can associate its Young diagram, which is a collection of lattice squares in the first quadrant. For example, the following is the Young diagram of the partition $(4,3,1)$:

\begin{picture}(100,160)(-90,-15)
\label{fig}


\put(0,0){\line(1,0){160}}
\put(0,40){\line(1,0){160}}
\put(0,80){\line(1,0){120}}
\put(0,120){\line(1,0){40}}

\put(0,0){\line(0,1){120}}
\put(40,0){\line(0,1){120}}
\put(80,0){\line(0,1){80}}
\put(120,0){\line(0,1){80}}
\put(160,0){\line(0,1){40}}

\put(160,40){\circle*{5}}
\put(120,80){\circle*{5}}
\put(40,120){\circle*{5}}

\put(160,0){\circle{5}}
\put(120,40){\circle{5}}
\put(40,80){\circle{5}}
\put(0,120){\circle{5}}

\put(162,3){\scriptsize{$(4,0)$}}
\put(162,43){\scriptsize{$(4,1)$}}
\put(122,43){\scriptsize{$(3,1)$}}
\put(122,83){\scriptsize{$(3,2)$}}
\put(42,83){\scriptsize{$(1,2)$}}
\put(42,123){\scriptsize{$(1,3)$}}
\put(2,123){\scriptsize{$(0,3)$}}

\put(65,-20){\mbox{Figure \ref{fig}}}

\end{picture}

\text{}\\
The hollow circles indicate the \textbf{inner} corners of the partition, while the solid circles indicate the \textbf{outer} corners. Given two partitions, we will write $\lambda \leq \mu$ if the Young diagram of $\lambda$ is completely contained in that of $\mu$. A standard Young tableau (denoted by $\syt$, plural $\sytx$) of shape $\mu - \lambda$ is an arrangement of the numbers $1,2,...,k$ in the boxes of $\mu - \lambda$, in such a way that the numbers decrease as we go up or to the right. There is a bijection between $\sytx$ and all collections of intermediary partitions:
$$
\lambda = \rho_k \leq \rho_{k-1} \leq... \leq \rho_{1} \leq \rho_{0} = \mu
$$
such that each partition $\rho_i$ has size one more than $\rho_{i+1}$. \\

\subsection{} 
\label{sub:fixy}

Formulas \eqref{eqn:formulam} are given in terms of tautological classes, but they can also be expressed in terms of torus fixed points. In our case, the torus action of $T$ on the smooth variety $\CM_d$ has finitely many fixed points, and these determine a linear basis of the $K-$theory group $K_d$. Namely, we have the following \textbf{localization formula}:
\begin{equation}
\label{eqn:local}
c = \sum_{p \in \CM_d^T} [p] \cdot \frac {c|_p}{\Lambda(T_p\CM_d,1)}, \qquad \forall c\in K_d
\end{equation}
where $[p]$ denotes the class of the skyscraper sheaf above the point $p$. Therefore, the classes $[p]$ form a basis for the $K-$theory group, after tensoring it with $\text{Frac}(\kko)$. The torus fixed points of $\CM_d$ are indexed by $r-$tuples of partitions $\lambda =(\lambda^1,...,\lambda^r)$ whose sizes sum up $d$, and are given by:
$$
\CI_\lambda := \CI_{\lambda^1} \oplus ... \oplus \CI_{\lambda^r}
$$ 
In the above, for any partition $\lambda^i=(\lambda^i_0 \geq \lambda^i_1 \geq ...)$, we consider the monomial ideal $\CI_{\lambda^i} = (x^{\lambda^i_0}y^0, x^{\lambda^i_1}y^1,...)$. As seen in \eqref{eqn:local}, the following constants will be important:
\begin{equation}
\label{eqn:defg}
g_\lambda = \Lambda(T_\lambda \CM_d,1) \in \kk
\end{equation}
While there are many explicit formulas for the character of $T$ in the tangent space to $\CM_d$, and hence also for $g_\lambda$, these constants have an important combinatorial meaning. \\

\subsection{} 
\label{sub:gen}

We will often apply the language of partitions to $r-$tuples of partitions. Namely, a square or corner in an $r-$tuple will simply be a square or corner in one of its constituent partitions. For $r-$tuples of partitions $\lambda$ and $\mu$, a SYT of shape $\mu-\lambda$ is a way to fill the boxes of this $r-$tuple of skew diagrams with the numbers $1,...,k$, in such a way that the numbers decrease as we go up or to the right. Given an $r-$tuple of partitions $\lambda$, the \textbf{weight} of a square with lower left corner $(i,j)$ is defined to be:
$$
\chi(\square) = q_1^iq_2^jt_k^{-1}
$$ 
where $k\in \{1,...,r\}$ indicates which partition the square lies in. The character of $T$ acting in the fibers of the tautological bundle $\CT$ is given in terms of these weights: 
\begin{equation}
\label{eqn:aaron}
\textrm{char}_\tT(\CT|_\lambda) = \sum_{\square \in \lambda} \chi(\square) 
\end{equation}
where the sum goes over all the boxes in the $r$ constituent partitions of $\lambda$. Therefore, the character in the fibers of $\CW$ is:
$$
\textrm{char}_\tT(\CW|_\lambda) = \sum_{k=1}^r t_k^{-1} - (1-q_1)(1-q_2) \sum_{\square \in \lambda} \chi(\square)  =
$$

\begin{equation}
\label{eqn:nu}
= \sum^{\square \textrm{ inner}}_{\textrm{corner of }\lambda} \chi(\square) - \sum^{\square \textrm{ outer}}_{\textrm{corner of }\lambda} \chi(\square)
\end{equation}

\text{} \\

\begin{proof} \textbf{of Proposition \ref{prop:gen}:} Note from \eqref{eqn:aaron} that the class $[\CT]$ of the tautological bundle has different restrictions to all torus fixed points. Since the Vandermonde determinant is non-zero, the class $[\lambda]$ of any torus fixed point can be written as a combination of the powers of $[\CT]$. Then \eqref{eqn:local} implies that any class in $K_d$ can be written as a combination of these powers. 

\end{proof}

\subsection{} In the proof of Theorem \ref{thm:main}, we have showed how the shuffle elements $X_m^\pm$ (or alternatively, the geometric correspondences $x^\pm_m$) act on tautological classes $\Lambda_d^S$. In this section, we will rephrase those computations in the basis of torus fixed points, which will give rise to a new interpretation of formulas \eqref{eqn:formulam}. \\

\begin{proposition}
\label{prop:zhenya}



For any Laurent polynomial $m$, the matrix coefficients of the operators $X_m^\pm$ in the torus fixed point basis $[\lambda]$ are given by:
$$
\langle \mu | X^\pm_m | \lambda \rangle = \frac {g_\lambda}{g_\mu} \sum^{\emph{SYT}}_{\emph{shape } \pm\mu \mp \lambda}   \frac {m(\chi_1,...,\chi_k) \prod_{i<j} \omega \left( \frac {\chi_i}{\chi_j} \right)}{\left(1-\frac {\chi_2q}{\chi_1} \right) ... \left(1-\frac {\chi_{k}q}{\chi_{k-1}} \right)}   \prod_{i=1}^k \chi_i^{-r\e} \Lambda(\pm \CW|_\mu^\pm, \chi^{\pm 1}_iq^\e)
$$
where $\chi_1,...,\chi_k$ denote the weights of the squares labeled $1,2,...,k$ inside each standard Young tableau. \\


\end{proposition}

\begin{remark}
\label{rem:rem}  Note that each summand in the RHS of the above expressions has precisely $k$ linear factors which vanish in the denominator. These factors have to be removed in order for the corresponding summand to make sense (alternatively, one can multiply the RHS by $0^k$). The reason for this will be apparent in the subsequent proof, which will compute the integrals \eqref{eqn:formulam} via residues. Essentially, we will repeatedly use identities of the form:
\begin{equation}
\label{eqn:onegin}
\int \frac {\prod_i (u-a_i)}{\prod_j (u-b_j)}Du = \sum_k \frac {\prod_i (b_k-a_i)}{\prod_{j\neq k} (b_k-b_j)}
\end{equation}
Under this analogy, the quantity displayed in the RHS of Proposition \ref{prop:zhenya} would be:
$$
\sum_k \frac {\prod_i (b_k-a_i)}{\prod_{\text{all } j} (b_k-b_j)}
$$
so the terms $b_k-b_k$ need to be removed from the denominators in order to obtain the correct RHS of \eqref{eqn:onegin}. \\

\end{remark}

\begin{proof} The localization formula \eqref{eqn:local} implies that:
$$
\L_d^S = \sum_{\lambda \vdash d} \frac {[\lambda]}{g_\lambda} \prod_{s\in S}\prod_{\square \in \lambda} \left(1- \frac {s}{\chi(\square)} \right)
$$
\footnote{The notation $\lambda \vdash d$ means that $\lambda$ is an $r-$tuple of partitions whose sizes add up to $d$} and therefore \eqref{eqn:formulam} gives:
$$
X^\pm \cdot \Lambda^S_d =  \sum_{\lambda \vdash d} \frac {X^\pm_m \cdot [\lambda] }{g_\lambda}  \prod_{s\in S} \prod_{\square \in \lambda} \left(1- \frac {s}{\chi(\square)} \right) =  \sum_{\mu \vdash d \pm k} \frac {[\mu] }{g_\mu}  \prod_{s\in S} \prod_{\square \in \mu} \left(1- \frac { s}{\chi(\square)} \right)
$$

\begin{equation}
\label{eqn:integral}
(\pm 1)^k  \int \frac {m(u_1,...,u_k) \prod_{i<j} \omega \left( \frac {u_i}{u_j} \right)}{\left(1-\frac {u_2q}{u_1} \right) ... \left(1-\frac {u_{k}q}{u_{k-1}} \right)}   \prod_{i=1}^k \left[ \frac {\Lambda(\pm \CW|^{\pm}_\mu, u^{\pm 1}_iq^\e)}{u_i^{r\e}\prod_{s \in S} \left(1 - \frac s{u_i} \right)^{\pm 1}} Du_i \right] 
\end{equation}
with $u_1$/$u_k$ being closest to the set of poles of $\Lambda(\pm \CW|^{\pm}_\mu, z^{\pm 1}q^\e)$, depending on whether the sign is $+$ or $-$. Let us compute the above integral by summing over the residues at these poles. By looking at \eqref{eqn:nu}, we see that:
$$
\Lambda(\CW|_\mu, uq) = \frac {\prod^{\square \textrm{ inner}}_{\textrm{corner of }\mu} \left(1 - \frac {uq}{\chi(\square)} \right)}{ \prod^{\square \textrm{ outer}}_{\textrm{corner of }\mu} \left(1 - \frac {uq}{\chi(\square)}\right)}, \quad \Lambda(-\CW|^\vee_\mu, u^{-1}) = \frac { \prod^{\square \textrm{ outer}}_{\textrm{corner of }\mu} \left(1 - \frac {\chi(\square)}u\right)}{\prod^{\square \textrm{ inner}}_{\textrm{corner of }\mu} \left(1 - \frac {\chi(\square)}u \right)}
$$
When the sign is $+$ and we integrate over $u_1$, we pick up a residue whenever $u_1q$ equals the weight of some outer corner of the partition $\mu$. When the sign is $-$ and we integrate over $u_k$, we pick up a residue whenever $u_k$ equals the weight of some inner corner of the partition $\mu$. If the sign is $+$ (respectively $-$), let $\square_1$ (respectively $\square_k$) be the square whose weight is $\chi_1:=\chi(\square_1) = u_1$ (respectively $\chi_k:=\chi(\square_k) = u_k$). We may remove (when the sign is $+$) or add (when the sign is $-$) this square from/to the partition, and then $\rho_{1} = \mu - \square_1$ (respectively $\rho_{k-1} = \mu+\square_k$) is a partition in its own right. The integral in \eqref{eqn:integral} then becomes: 
$$
(\pm 1)^{k-1}  \sum^{\rho_{1} \leq \mu \text{ or}}_{\rho_{k-1} \geq \mu}  \int \frac {m(u_1,...,u_k) \prod_{i<j} \omega \left( \frac {u_i}{u_j} \right)}{\left(1-\frac {u_2q}{u_1} \right) ... \left(1-\frac {u_{k}q}{u_{k-1}} \right)}   \prod_{i=1}^k \left[ \frac {\Lambda(\pm \CW|^\pm_\mu, u_i^{\pm 1}q^\e)}{u_i^{r\e}\prod_{s \in S} \left(1 - \frac s{u_i} \right)^{\pm 1}} Du_i \right] \Big |^{u_1 = \chi_1 \text{ or}}_{u_k = \chi_k}
$$
Now we need to integrate over $u_2$ (respetively $u_{k-1}$). When the sign is $+$, we pick up poles when either $u_2q$ is the weight of some outer square of $\mu$, or $u_2=\chi_1q_1^{-1}$ or $u_2=\chi_1q_1^{-1}$. When the sign is $-$, we pick up poles when either $u_{k-1}$ is the weight of some inner square of $\mu$, or $u_{k-1}=\chi_kq_1^{-1}$ or $u_{k-1}=\chi_k q_1^{-1}$. In either of these cases, note that $u_2=\chi_1$ (respectively $u_{k-1}=\chi_k$) is not a viable option for a pole anymore, because the numerator of $\omega$ eliminates it. If the sign is $+$ (respectively $-$), let $\square_2$ (respectively $\square_{k-1}$) be the square whose weight is $\chi_2:=\chi(\square_2) = u_2$ (respectively $\chi_{k-1}:=\chi(\square_{k-1}) = u_{k-1}$). Note that $\rho_2 = \rho_1 - \square_2$ (respectively $\rho_{k-2} = \rho_{k-1}+\square_{k-1}$) is also a partition. We conclude that the integral \eqref{eqn:integral} equals: 
$$
(\pm 1)^{k-2}  \sum^{\rho_{2} \leq \rho_1 \leq \mu \text{ or}}_{\rho_{k-2} \geq \rho_{k-1} \geq \mu}  \int \frac {m(u_1,...,u_k) \prod_{i<j} \omega \left( \frac {u_i}{u_j} \right)}{\left(1-\frac {u_2q}{u_1} \right) ... \left(1-\frac {u_{k}q}{u_{k-1}} \right)}   \prod_{i=1}^k \left[ \frac {\Lambda(\pm \CW|^\pm_\mu, u^{\pm 1}_iq^\e)}{u_i^{r\e}\prod_{s \in S} \left(1 - \frac s{u_i} \right)^{\pm 1}} Du_i \right]
$$
evaluated at $u_1 = \chi_1, u_2 = \chi_2$ (when the sign is $+$) or  $u_k = \chi_k, u_{k-1}=\chi_{k-1}$ (when the sign is $-$). Repeating the above procedure for the remaining integrals gives us the following result for the integral \eqref{eqn:integral}:
$$
  \sum^{\lambda \leq \rho_{k-1} \leq ...  \leq \rho_1 \leq \mu \text{ or}}_{\lambda \geq \rho_{1} \geq ...  \geq \rho_{k-1} \geq \mu}   \quad \frac {m(\chi_1,...,\chi_k) \prod_{i<j} \omega \left( \frac {\chi_i}{\chi_j} \right)}{\left(1-\frac {\chi_2q}{\chi_1} \right) ... \left(1-\frac {\chi_{k}q}{\chi_{k-1}} \right)}   \prod_{i=1}^k\frac {\Lambda(\pm \CW|^\pm_\mu, \chi^{\pm 1}_i q^\e)}{\chi_i^{r\e}\prod_{s \in S} \left(1 - \frac s{\chi_i} \right)^{\pm 1}}
$$
where $\chi_i$ is the weight of the square $\rho_{i-1} - \rho_{i}$. Since such a flag of partitions precisely determines a SYT, we conclude that \eqref{eqn:integral} implies: 
$$
\sum_{\lambda \vdash d} \frac {X^\pm_m \cdot [\lambda] }{g_\lambda}  \prod_{s\in S} \prod_{\square \in \lambda} \left(1- \frac {s}{\chi(\square)} \right) =  \sum_{\mu \vdash d+k} \frac {[\mu] }{g_\mu}  \prod_{s\in S} \prod_{\square \in \mu} \left(1- \frac { s}{\chi(\square)} \right)
$$

$$
 \sum_{\syt\text{ of shape } \pm \mu \mp \lambda}   \frac {m(\chi_1,...,\chi_k) \prod_{i<j} \omega \left( \frac {\chi_i}{\chi_j} \right)}{\left(1-\frac {\chi_2q}{\chi_1} \right) ... \left(1-\frac {\chi_kq}{\chi_{k-1}} \right)}   \prod_{i=1}^k\frac {\Lambda(\pm \CW|^\pm_\mu, \chi^{\pm 1}_i q^\e)}{\chi_i^{r\e}\prod_{s \in S} \left(1 - \frac s{\chi_i} \right)^{\pm 1}}
$$
Since the above relation must hold for all sets of variables $S$, this implies that they hold for each $\lambda$ individually after canceling all the terms that contain $s\in S$. 

\end{proof}

\textbf{} \\
Using \eqref{eqn:heis}, the above Proposition tells us how to compute the matrix coefficients of the operators $P_{\pm k, d}$ acting on $K$ in the basis of torus fixed points, for all $k\neq 0$. The matrix coefficients of $P_{0,\pm n} = p^\pm_n$ are simply the $\pm n-$th power sums of the weights of the boxes at the fixed point in question. \\

\section{Other Geometric Constructions}
\label{sec:extra}

\subsection{}
\label{sub:chare}

We will consider the vector bundle $E$ on $\CM_{d^-} \times \CM_{d^+}$, whose fibers are:
$$
E|_{\CF_-,\CF_+} = \Ext^1(\CF_+,\CF_-(-\infty))
$$
As a class in $K-$theory, it can be written as:
\begin{equation}
\label{eqn:tan}
[E] = p^{\mp*} \left([T \CM]\right) \pm q^{-\e} p^{\mp *} \left([\CW]^{\pm}\right) \cdot \left[ p^{+*}\left([\CT]\right)-p^{-*}\left([\CT] \right) \right]^{\mp}
\end{equation}
where $p^+,p^-:\CM \times \CM \longrightarrow \CM$ are the projections to the two factors. We will show how to prove the above formula in Section \ref{sub:proofe}. Note that by \eqref{eqn:ks}, we see that the restriction of $E$ to the diagonal $\Delta \subset \CM_d \times \CM_d$ is precisely the tangent space to $\CM_d$. Therefore, \eqref{eqn:tan} also gives us a formula for the character in these tangent spaces, and thus a formula for computing the constants $g_\lambda$ of \eqref{eqn:defg}. Consider the long exact sequence:
$$
\Hom(\CF_+,\CF_-) \longrightarrow \Hom(\CF_+,\CF_-|_\infty) \longrightarrow \Ext^1(\CF_+,\CF_-(-\infty))
$$ 
We have the tautological map $\CF_+ \longrightarrow \CF_+|_\infty \cong \CF_-|_\infty$, viewed as an element in the middle $\Hom$ space. Pushing this element to $\Ext^1$ gives us a section: 
\begin{equation}
\label{eqn:section}
s\in \Gamma(\CM_d, E)
\end{equation}
It is easy to see from the above exact sequence that this section vanishes precisely when $\CF_+ \subset \CF_-$. \\

\subsection{} 
\label{sub:defv}

Just like the moduli of flags of sheaves $\fZ^k$, many constructions consisting of nested sheaves are singular. The main exception is the variety:
\begin{equation}
\label{eqn:usl}
\fV^k = \{(\CF,\CF') \text{ such that }\CF \supset \CF' \supset \CF(-\nu)\} \subset \CM_d \times \CM_{d+k}
\end{equation}
where $\nu = \{y=0\}$ is a line in $\BP^2$. The above is a particular type of smooth moduli space of parabolic sheaves, and it is well-known to be Lagrangian inside the product of symplectic varieties $\CM_d \times \CM_{d+k}$. Alternatively, we will show in Section \ref{sub:proofv} that $\fV^k$ can be regarded as the fixed locus of a $\BZ/2\BZ-$action on the moduli space $\CM_{2d+k}$ and this will allow us to compute its tangent space:
\begin{equation}
\label{eqn:ressult}
[T \fV^k] = p^{\pm *} ([T \CM]) \mp q^{-\e} p^{\pm *}([\CW]^{\pm}) \cdot l^{\mp} - (1-q_1^{-1}) \cdot l \cdot l^\vee 
\end{equation}
Here, $q_1$ is the equivariant parameter in the direction of the line $\nu$, and $l = [\CL]$ is the $K-$theory class of the tautological rank $k$ vector bundle on $\fV^k$:
$$
\CL|_{\CF \supset \CF'} = \rg(\BP^2, \CF/\CF')
$$
We will show how to prove \eqref{eqn:ressult} in Section \ref{sub:proofv}. \\

\subsection{} These constructions give rise to functors on the derived categories of $\CM_d$:
\begin{equation}
\label{eqn:der}
\CC_\bullet \longrightarrow \CC_{\bullet \pm k},  \qquad c \longrightarrow Rp^\pm_{*}(K_s(E) \otimes p^{\mp *}(c))
\end{equation}

$$
\CC_\bullet \longrightarrow \CC_{\bullet\pm k}, \quad \qquad c \longrightarrow Rp^\pm_{*}(\CO_{\fV^k} \otimes p^{\mp *}(c))
$$
where $p^-, p^+: \CM \times \CM \longrightarrow \CM$ are the projections onto the first and second factors. Here, $K_s(E)$ denotes the Koszul complex of the vector bundle $E^\vee$ with respect to the section $s$ of \eqref{eqn:section}:
$$
[\Lambda^{\text{rk }E}E^\vee \longrightarrow ... \longrightarrow \Lambda^2 E^\vee \longrightarrow E^\vee \stackrel{s^\vee} \longrightarrow \CO_{\CM_{d^-} \times \CM_{d^+}}]
$$ 
At the level of $K-$theory, the above functors give rise to linear operators:
$$
a_k^\pm : K_\bullet \longrightarrow K_{\bullet \pm k}, \quad \ \qquad c \longrightarrow p^\pm_{*}(\Lambda(E,1) \cdot p^{\mp *}(c))
$$

\begin{equation}
\label{eqn:ex}
b_k^\pm: K_\bullet \longrightarrow K_{\bullet \pm k}, \qquad \ \qquad c \longrightarrow p^\pm_{*}(\CO_{\fV^k} \cdot p^{\mp *}(c))
\end{equation}
In the remainder of this paper, we will compute the above operators in terms of the shuffle algebra. \footnote{We have chosen the Koszul complex in \eqref{eqn:der} because it induces the class $\Lambda(E)$ in $K-$theory. As for the choice of the section $s$ with respect to which the complex is defined, this was chosen so that when $d^+=d^-+1$, \eqref{eqn:der} coincides with \eqref{eqn:derived} for $m=1$. Indeed, in that case, the section $s$ scheme-theoretically cuts out the correspondence $\fZ_{d,d+1}$} \\

\subsection{} 
\label{sub:compe}

The following elements of the shuffle algebra were defined in \cite{FHHSY}:
\begin{equation}
\label{eqn:formulaa}
A_k = \frac {(1-q)^k}{(1-q_1)^k(1-q_2)^k}  \prod_{1\leq i \neq j \leq k} \frac {(z_i-z_j)(z_i-qz_j)}{(z_i-q_1 z_j)(z_i-q_2 z_j)} \in \CA^+ 
\end{equation}

\begin{equation}
\label{eqn:formulab}
B_k = \frac {q_1^{\frac {k(k-1)}2}}{(1-q_1)^k} \prod_{1\leq i \neq j \leq k} \frac {z_i-z_j}{z_i-q_1 z_j} \in \CA^+
\end{equation}
As shown in \cite{FHHSY}, \cite{Shuf}, the above elements lie in the commutative subalgebra generated by $\{P_{1,0},P_{2,0},...\} \subset \CA^+$. It is also very easy to compute their coproduct, which was described in \cite{Shuf} and denoted therein by $\Delta_0$:
$$
\Delta_0(A_k) = \sum_{i=0}^k H_0^{k-i} A_i \otimes A_{k-i}, \qquad \Delta_0(B_k) = \sum_{i=1}^k H_0^{k-i} B_i \otimes B_{k-i}
$$
Elements with the above coproduct are called group-like, and they are always exponentials of the $q-$Heisenberg generators:
\begin{equation}
\label{eqn:vertex}
\sum_{k=0}^\infty A_kz^k = \exp\left(\sum_{k=1}^\infty \alpha_k \frac  { P_{k,0} z^k}k  \right), \qquad  \sum_{k=0}^\infty B_kz^k = \exp\left(\sum_{k=1}^\infty \beta_k  \frac  {P_{k,0} z^k}k  \right)
\end{equation}
where $\alpha_k, \beta_k \in \text{Frac}(\kk)$ are some constants. \\


\subsection{} To determine these constants, we will use the multiplicative linear map $\ph$ of \cite{Shuf} (we set $d=0$ in the notation of $\emph{loc. cit.}$):
$$
\ph: \langle P_{1,0},P_{2,0},... \rangle \longrightarrow \text{Frac}(\kk) 
$$

$$
\ph(R) = \left[ R(z_1,...,z_k) \prod_{1\leq i \neq j \leq k} \frac {z_i-q_1z_j}{z_i-z_j} \right]_{z_i = q_1^{-i}}  \frac {q_1^{-\frac {k^2}2+k}}{(1-q_2)^k} \prod_{i=1}^k \frac {q_1^{i-1}-q_2}{q_1^i-1}
$$
It is easy to see that:
$$
\ph(A_k) = \frac {q_1^{\frac k2}}{(1-q_1)^{k}(1-q_2)^{k}} \prod_{i=1}^k \frac {1-q_1^{i}q_2}{q_1^i-1}
$$




$$
\ph(B_k) = \frac {q_1^{\frac k2}}{(1-q_1)^{k}(1-q_2)^{k}} \prod_{i=1}^k \frac {q_1^{i-1} - q_2}{q_1^i-1}
$$
while: 
$$
\ph(P_{k,0}) =\frac {(-1)^k q_1^{\frac k2}(1-q_2^k)}{(1-q_1)^k(1-q_2)^k}
$$ 
was computed in 6.12 of \cite{Shuf} \footnote{The discrepancy between the above formula and \loccit is due to the fact that we have renormalized our $P_{k,0}$}. Since $\ph$ is multiplicative, plugging these identities in \eqref{eqn:vertex} gives us:
$$
\alpha_k = \frac {(1-q^k)}{(1-q_1^k)(1-q_2^k)}, \qquad \beta_k = \frac {(-1)^{k-1}}{1-q_1^k}
$$
so we conclude that:
$$
\sum_{k=0}^\infty A_kz^k = \exp \left(\sum_{k=1}^\infty \frac {(1-q^k)}{(1-q_1^k)(1-q_2^k)}\cdot  \frac {P_{k,0} z^k}k\right)
$$

$$
\sum_{k=0}^\infty B_kz^k = \exp \left(\sum_{k=1}^\infty \frac {(-1)^{k-1}}{1-q_1^k} \cdot \frac {P_{k,0} z^k}k \right)
$$

\subsection{} We will now show that the above shuffle elements act on $K$ via the geometric operators of \eqref{eqn:ex}: \\

\begin{proposition}
\label{prop:extra}

As endomorphisms of $K$, we have:
\begin{equation}
\label{eqn:boris1}
a^\pm_k = A^\pm_k(z_1,...,z_k)\cdot (z_1...z_k)^{r\e}
\end{equation}
\begin{equation}
\label{eqn:boris2}
b^\pm_k = q_1^{-\frac {k(k+1)}2} B^\pm_k(z_1,...,z_k) \cdot (z_1...z_k)^{r\e}
\end{equation}
where recall that $\e$ is $1$ or $0$ depending on whether the sign is $+$ or $-$. \\
\end{proposition}

\begin{proof} By the equivariant localization formula \eqref{eqn:local}, we have:
\begin{equation}
\label{eqn:ytu}
a^\pm_k \cdot \Lambda^S_d = \sum_{\lambda_+,\lambda_- \in \CM^T}  \frac {[\lambda_\pm]}{g_{\lambda^\pm}} \cdot  \frac{\Lambda(E_{\lambda_-,\lambda_+},1)}{\Lambda(T_{\lambda_{\mp}}\CM,1)} \cdot \Lambda^S_{\lambda^\mp}
\end{equation}
Let us remark that $E_{\lambda^-,\lambda^+}$ contains a trivial character 1, and so the above numerator vanishes, unless $\lambda_- \leq \lambda_+$. The reason for this is that the section $s\in \Gamma(E)$ vanishes on the locus $\CF^- \supset \CF^+$. Therefore, \eqref{eqn:ytu} becomes:
\begin{equation}
\label{eqn:ytu2}
a^\pm_k \cdot \Lambda^S_d = \sum_{\lambda_\pm \in \CM^T} \frac {[\lambda_\pm]\cdot \Lambda^S_{\lambda^\pm}}{g_{\lambda_{\pm}}} \sum_{\lambda_- \leq \lambda_+} \frac{\Lambda(E_{\lambda_-,\lambda_+})}{\Lambda(T_{\lambda_{\mp}}\CM)} \prod_{s\in S} \prod_{\square \in \lambda^+-\lambda^-}\left(1 - \frac s{\chi(\square)} \right)^{\mp 1}
\end{equation}
From \eqref{eqn:tan}, we infer that:
$$
[E_{\lambda_-,\lambda_+}] - [T_{\lambda^\mp}\CM] = \pm q^{-\e} \CW^\pm_{\lambda^{\mp}} \cdot ([\CT_{\lambda_+}]-[\CT_{\lambda_-}])^{\mp} = \pm q^{-\e} \CW^\pm_{\lambda^{\mp}}\cdot \sum_{\square \in \lambda_+-\lambda_-} \chi(\square)^{\mp 1} =
$$
\begin{equation}
\label{eqn:sheldon}
= \pm q^{-\e} \CW^\pm_{\lambda^{\pm }}\cdot \sum_{\square \in \lambda_+-\lambda_-} \chi(\square)^{\mp 1} + (1-q^{-1}_1)(1-q^{-1}_2)\sum_{\square, \square' \in \lambda_+-\lambda_-} \frac {\chi(\square)}{\chi(\square')} \qquad
\end{equation}
and therefore, \eqref{eqn:ytu2} becomes:
$$
 = \sum_{\lambda_\pm \in \CM^T} \frac {[\lambda_\pm]\cdot \Lambda^S_{\lambda^{\pm}}}{g_{\lambda_{\pm}}}  \cdot \sum_{\lambda_- \leq \lambda_+} \prod_{\square,\square' \in \lambda_+ - \lambda_-} \omega\left(\frac {\chi(\square)}{\chi(\square')} \right)  \prod_{\square \in \lambda_+-\lambda_-} \frac {\Lambda(\pm \CW^{\pm}_{\lambda^\pm},\chi(\square)^{\pm 1}q^\e)}{\prod_{s\in S}  \left(1-\frac s{\chi(\square)}\right)^{\pm 1}} 
$$
The above sum goes over all $\lambda_- \leq \lambda_+$ with $|\lambda_+| - |\lambda_-| = k$. If we fix the partition $\lambda_+$ (or $\lambda_-$), the sum goes over all the ways to remove (or add, respectively) $k$ non-ordered boxes $\square$ from this partition. We claim that the corresponding $\chi(\square)$ are precisely the poles of a rational function, in that the above relation becomes:
$$
a^\pm_k \cdot \Lambda^S_{\lambda^{\mp}} = \sum_{\lambda_\pm \in \CM^T} \frac {[\lambda_\pm]\cdot \Lambda^S_{\lambda^{\pm}}}{g_{\lambda^{\pm}}}    :\int: \prod_{1\leq i , j \leq k} \omega\left(\frac {u_i}{u_j} \right) \prod_{i=1}^k \left[\frac {\Lambda(\pm \CW^{\pm}_{\lambda^\pm}, u_i^{\pm 1}q^{\e})}{\prod_{s\in S} \left(1-\frac s{u_i} \right)^{\pm 1}}  Du_i \right]
$$
The reason we need to take the normal ordered integral of Section \ref{sub:normal} is that we must count each configuration of added/removed boxes exactly once. We need to remove the zeroes $u_i-u_i$ from the numerator of $\omega(u_i/u_i)$, since they precisely account for the residue computation (see Remark \ref{rem:rem}). Delocalizing the above, we see that:
$$
a^\pm_k \cdot \Lambda^S_d = \Lambda^S_{d\pm k} \left[ \frac {\pm(1-q)}{(1-q_1)(1-q_2)} \right]^k :\int: \prod_{1\leq i \neq j \leq k} \omega\left(\frac {u_i}{u_j} \right) \prod_{i=1}^k \left[\frac {\Lambda(\pm \CW^{\pm}, u_i^{\pm 1}q^{\e})}{\prod_{s\in S} \left(1-\frac s{u_i} \right)^{\pm 1}}  Du_i \right]
$$
Comparing this with the formulas for $A^\pm_k$ given by \eqref{eqn:formulaa} proves \eqref{eqn:boris1}. As for \eqref{eqn:boris2},  we need to play the same game:
\begin{equation}
\label{eqn:leonard}
b^\pm_k \cdot \Lambda^S_d = \sum_{\lambda_- \leq \lambda_+}  \frac {[\lambda_\pm]}{\Lambda(T_{(\lambda_-,\lambda_+)}\fV^k,1)} \prod_{s\in S} \Lambda(\CT_{\lambda^{\mp}},s)
\end{equation}
We can use relation \eqref{eqn:ressult} to compute:
$$
[T_{\lambda^\pm}\CM] - [T_{(\lambda^-,\lambda^+)} \fV^k]  = \pm q^{-\e} [\CW^{\pm}_{\lambda^\pm}] \sum_{\square \in \lambda_+ - \lambda_-} \chi(\square)^{\mp} + (1-q^{-1}_1) \sum_{\square, \square' \in \lambda_+\backslash \lambda_-} \frac {\chi(\square)}{\chi(\square')}
$$
This formula differs from \eqref{eqn:sheldon} only in the coefficient in front of the last term, which is $1-q^{-1}_1$ instead of $(1-q_1^{-1})(1-q^{-1}_2)$. Therefore, the whole discussion that applied to $a_k^\pm$ allows to write \eqref{eqn:leonard} as the normal ordered integral:
$$
b^\pm_k \cdot \Lambda^S_d = \Lambda^S_{d\pm k} \left[ \frac {\pm 1}{1-q_1} \right]^k :\int: \prod_{1 \leq i\neq j \leq k}  \frac {u_i-u_j}{u_iq_1-u_j}  \prod_{i=1}^k \left[\frac {\Lambda(\pm \CW^{\pm}, u_i^{\pm 1}q^{\e})}{\prod_{s\in S} \left(1-\frac s{u_i} \right)^{\pm 1}}  Du_i \right]
$$
Comparing this with the formulas for $B_k^\pm$ given by \eqref{eqn:formulab} proves \eqref{eqn:boris2}. 

\end{proof}

\subsection{}
\label{sub:proofe}

Formula \eqref{eqn:tan} follows from the Proposition below: \\


\begin{proposition}

We have the following equality of $K-$theory classes:
\begin{equation}
\label{eqn:master}
[E] = \sum_{i=1}^r t_i \cdot  [\CT_-] + \sum_{i=1}^r q^{-1}t_i^{-1} \cdot [\CT_+]^\vee - (1-q_1^{-1})(1-q_2^{-1}) \cdot [\CT_-] \cdot [\CT_+]^\vee \qquad
\end{equation}
where $\CT_-$ and $\CT_+$ denote the tautological bundles on the first and second, respectively, factors of $\CM \times \CM$. \\

\end{proposition}

\begin{proof} By the localization formula \eqref{eqn:local}, it is enough to prove that both sides of \eqref{eqn:master} have the same restriction to all torus fixed points. We have:
$$
[E_{\lambda,\mu}] = \sum_{i=1}^r \sum_{i'=1}^r \frac {t_i}{t_{i'}} \cdot [\Ext^1(\CI_{\mu^{i}}, \CI_{\lambda^{i'}}(-\infty))] 
$$
where the last equality holds because $\Ext^0$ and $\Ext^2$ vanish, the former because of the twist by $\infty$, and the latter because of Serre duality. The characters in the above $\Ext^1$ spaces have been computed, for example, in Lemma 4.14 of \cite{FFNR}:
$$
[E_{\lambda,\mu}] =  \sum_{i=1}^r \sum_{i'=1}^r \frac {t_{i}}{t_{i'}} \left(\sum_{j'\geq 0} \frac {q_1^{\lambda^{i'}_{j'}}-1}{q_1-1} q_2^{j'} + q^{-1} \sum_{j\geq 0} \frac {q_1^{-\mu^{i}_{j}}-1}{q^{-1}_1-1} q_2^{-j} - \right.
$$
$$
\left. - (1-q_1^{-1})(1-q_2^{-1})\sum_{j\geq 0} \sum_{j'\geq 0} \frac {(q_1^{\lambda^{i'}_{j'}}-1)(q_1^{-\mu^i_j}-1)}{(q_1-1)(q_1^{-1}-1)} q_2^{j'-j} \right)
$$
Comparing this with the character in the tautological sheaves from \eqref{eqn:aaron} gives us:
$$
[E_{\lambda,\mu}] = \sum_{i=1}^r t_i \cdot [\CT_\lambda] + \sum_{i=1}^r q^{-1} t^{-1}_i \cdot [\CT_\mu]^\vee - (1-q_1^{-1})(1-q^{-1}_2) \cdot [\CT_\lambda]\cdot [\CT_\mu]^\vee
$$
Delocalizing the above relation gives us precisely \eqref{eqn:master}.



\end{proof}

\subsection{}
\label{sub:proofv}

Let us now look at the varieties $\fV^k$ of Subsection \ref{sub:defv}, and show that they are smooth and compute the character \eqref{eqn:ressult} in their tangent spaces. We can interpret a flag of two sheaves \eqref{eqn:usl} as a single sheaf:
\begin{equation}
\label{eqn:victor}
\left(\CF \supset \CF' \supset \CF(-\nu) \right) \qquad \leftrightarrow \qquad \tCF = \tau^*(\CF')+\tau^*(\CF)(-\nu)
\end{equation}
where: 
$$
\tau:\BC^2 \longrightarrow \BC^2, \qquad \tau(x,y)=(x,y^2)
$$
This makes $\tCF$ into a sheaf on $\BC^2$, but we can glue a trivial sheaf at $\infty$ to make it into a trivialized sheaf on $\BP^2$. Moreover, $\tCF$ is $\BZ/2\BZ-$invariant under the action of $\BZ/2\BZ$ on $\BP^2$ given by $[x:y:z] \longrightarrow [x;-y;z]$. Therefore, we have a map:
$$
\fV_{d,d+k} \longrightarrow \CM_{2d+k}^{\BZ/2\BZ}
$$
This map is an isomorphism onto a certain connected component of the $\BZ/2\BZ-$fixed locus on the right. This allows us to compute the tangent spaces to $\fV^k$ as the $\BZ/2\BZ-$fixed loci of the tangent spaces to $\CM$:
\begin{equation}
\label{eqn:inc}
[T_{\CF \supset \CF'}\fV^k] = \BZ/2\BZ-\text{invariant part of } [T_{\tCF} \CM]
\end{equation}
The $K-$theory class of $[T_{\tCF} \CM]$ is given by restricting \eqref{eqn:master} to the diagonal of $\CM$:
\begin{equation}
\label{eqn:tanmod}
[T\CM] = \sum_{i=1}^r t_i \cdot  [\tCT] + \sum_{i=1}^r q^{-1}t_i^{-1} \cdot [\tCT]^\vee - (1-q_1^{-1})(1-q_2^{-1}) \cdot [\tCT] \cdot [\tCT]^\vee \qquad
\end{equation}
where $\tCT$ denotes the tautological bundle on $\CM$. Under the inclusion \eqref{eqn:inc}, it is related to the tautological bundles on $\fV^k$ by: $[\tCT] = [\CT']+q_2[\CT]$. Therefore, \eqref{eqn:inc} and \eqref{eqn:tanmod} imply:
$$
[T \fV^k] = \BZ/2\BZ-\text{invariant part of } \sum_{i=1}^r t_i \cdot  \left([\CT'] + q_2[\CT]\right) + 
$$
$$
+\sum_{i=1}^r q^{-1}t_i^{-1} \cdot \left([\CT']^\vee + q_2^{-1} [\CT]^\vee \right) - (1-q_1^{-1})(1-q_2^{-1})  \left([\CT']+q_2[\CT]\right)  \left([\CT']^\vee+q^{-1}_2[\CT]^\vee\right)
$$
Taking the $\BZ/2\BZ-$invariant means only keeping those terms which contain $q_2^{2k}$ for some integer $k$, and replacing that with $q_2^k$. Therefore, the above relation gives us:
\begin{equation}
\label{eqn:tanlag}
[T \fV^k] = \sum_{i=1}^r t_i \cdot [\CT'] + \sum_{i=1}^r q^{-1}t_i^{-1} \cdot [\CT]^\vee   -
\end{equation}
$$
- (1-q_1^{-1})\left([\CT']\cdot [\CT']^\vee+ [\CT]\cdot [\CT]^\vee - [\CT]\cdot [\CT']^\vee - q_2^{-1}[\CT']\cdot [\CT]^\vee \right)
$$
Together with \eqref{eqn:tanmod}, formula \eqref{eqn:tanlag} implies \eqref{eqn:ressult}. 








\end{document}